\documentclass[11pt]{article}
\author{Vincenzo Ferrazzano
\thanks{Center for Mathematical Sciences, Technische Universit\"at M\"unchen,  85748 Garching b. M\"unchen, Germany,
email: ferrazzano@ma.tum.de, http://www-m4.ma.tum.de}\and
Florian Fuchs\thanks{Institute of Mathematical Finance, Ulm University, 89081 Ulm and International School of Applied Mathematics, Technische Universit\"at M\"unchen, 85748 Garching b. M\"unchen, Germany,
email: ffuchs@ma.tum.de, http://www-m4.ma.tum.de}
}

\title{Noise recovery for L\'evy-driven CARMA processes and high-frequency behaviour of approximating Riemann sums}

\usepackage{natbib}
\usepackage[latin9]{inputenc}
\usepackage[english]{babel}
\usepackage[T1]{fontenc}
\usepackage{amsmath,amssymb}
\usepackage{lscape}
\usepackage[final]{showkeys}
\usepackage{enumitem}
\usepackage{graphicx}
\usepackage{url}
\usepackage{latexsym}
\usepackage[usenames]{color}
\usepackage{float}
\usepackage{epsfig}
\usepackage{amssymb}
\usepackage{amsfonts}
\usepackage{amsmath}
\usepackage{algorithmic}
\usepackage{amsthm}

\allowdisplaybreaks 

\numberwithin{equation}{section}

\newtheorem{thm}{Theorem}[section]
\newtheorem{cor}[thm]{Corollary}
\newtheorem{lem}[thm]{Lemma}
\newtheorem{prop}[thm]{Proposition}
\newtheorem{defn}[thm]{Definition}
\newtheorem{example}[thm]{Example}
\newtheorem{Assumption}{Assumption}
\newtheorem{oss}[thm]{Remark}
\newtheorem{fig}[thm]{Figure}

\def\var{\mathop{\textrm{var}}}

\newcommand{\displayfrac}[2]{\frac{\displaystyle #1}{\displaystyle #2}}

\newcommand{\bthe}{\begin{thm}}
\newcommand{\ethe}{\end{thm}}

\newcommand{\ble}{\begin{lem}}
\newcommand{\ele}{\end{lem}}

\newcommand{\bde}{\begin{defn}}
\newcommand{\ede}{\end{defn}}

\newcommand{\bco}{\begin{cor}}
\newcommand{\eco}{\end{cor}}

\newcommand{\bpr}{\begin{prop}}
\newcommand{\epr}{\end{prop}}

\newcommand{\bproof}{\begin{proof}}
\newcommand{\eproof}{\end{proof}}

\newcommand{\bexam}{\begin{example}\rm}
\newcommand{\eexam}{\halmos\end{example}}

\newcommand{\brem}{\begin{oss}\rm}
\newcommand{\erem}{\halmos\end{oss}}

\newcommand{\bfi}{\begin{fig}}
\newcommand{\efi}{\end{fig}}

\newcommand{\beao}{\begin{eqnarray*}}
\newcommand{\eeao}{\end{eqnarray*}\noindent}

\newcommand{\beam}{\begin{eqnarray}}
\newcommand{\eeam}{\end{eqnarray}\noindent}

\newcommand{\barr}{\begin{array}}
\newcommand{\earr}{\end{array}}

\newcommand{\beq}{\begin{equation}}
\newcommand{\eeq}{\end{equation}}
\def\g{>}
\def\l{<}
\def\bbr{{\Bbb R}}
\def\bbn{{\Bbb N}}
\def\bbc{{\Bbb C}}
\def\bbz{{\Bbb Z}}
\def\bbe{{\Bbb E}}

\newcommand{\dsum}{\displaystyle\sum}

\newcommand{\stirlingtwo}[2]{\genfrac{\{}{\}}{0pt}{}{#1}{#2}}

\newcommand{\la}{{\lambda}}

\newcommand{\CARMA}{{\rm CARMA}}
\newcommand{\COGARCH}{{\rm COGARCH}}
\newcommand{\CAR}{{\rm CAR}}
\newcommand{\AR}{{\rm AR}}
\newcommand{\ARMA}{{\rm ARMA}}
\newcommand{\MA}{{\rm MA}}

\newcommand{\bone}{{\mathbf{1}}}

\def\halmos{\hfill $\Box$  \medskip }

\begin{document}
\maketitle
\begin{abstract} 
 We consider high-frequency sampled continuous-time autoregressive moving average \linebreak (CARMA) models driven by finite-variance zero-mean L\'evy processes. An $L^2$-consistent estimator for the increments of the driving L\'evy process without order selection in advance is proposed if the CARMA model is invertible. 
 In the second part we analyse the high-frequency behaviour of approximating Riemann sum processes, which represent a natural way to simulate continuous-time moving average models on a discrete grid. We compare their autocovariance structure with the one of sampled CARMA processes and show that the rule of integration plays a crucial role. Moreover, new insight into the kernel estimation procedure of \cite{bfk:2011:2} is given.
\end{abstract}

\vspace{0.2cm}

\noindent
\begin{tabbing}
\\ \> Secondary: 62M10
\end{tabbing}

\vspace{0.2cm}\noindent\emph{Keywords:} CARMA process, high-frequency data, L\'evy process, discretely sampled process, noise recovery.

\section{Introduction}
The constantly increasing availability of high-frequency data in finance and sciences in general has sparked in the last decade a great deal of attention about the asymptotic behaviour of high-frequency sampled {processes}, especially  concerning the estimation of multi-power variations of  It\=o {semimartingales} (see, e.g., \cite{realvolmultipow}, \cite{RPVstoch:vol}), employing their realised counterparts. These quantities are of primary importance to practitioners, since they embody the deviation of data from a Brownian motion. Such methods are summarised in the book of \cite{jacod:protter:discr}, which represents the most recent review on the subject.

In many areas of application L\'evy-driven processes are used for modelling time series. An ample class within this group are continuous-time moving average (CMA) processes
\begin{equation*}Y_t=\int_{-\infty}^\infty g(t-s)dL_s, \quad t \in\bbr,
\end{equation*}
%
where $g$ is the so-called kernel function and $L = \{L_t\}_{t\in\bbr}$ is said to be the driving L\'evy process (see, e.g., \cite{sato:1999} for a detailed introduction). They cover, for instance, Ornstein-Uhlenbeck and continuous-time autoregressive moving average (CARMA) processes. The latter are the continuous-time analogue of the well-known ARMA models (see, e.g., \cite{BD}) and have extensively been studied over the recent years (cf. \cite{brockwell5,Brockwell:2004,BrLi,todorov:tauchen:2006}). Originally, driving processes of CARMA models were restricted to Brownian motion (see \cite{Doob:1944}, and also \cite{Doob:1990df}). However, \cite{brockwell5} allowed for L\'evy processes with a finite $r$th moment for some $r>0$.

L\'evy-driven CARMA models are widely used in various areas of application like signal processing and control (cf. \cite{GarnierWang2008,LarssonMossbergSoederstroem2006}), high-frequency financial econometrics (cf. \cite{Todorov2009}), and financial mathematics (cf. \cite{BenthKoekebakkerZakamouline2010,BCL:2006,HaCz:2007,todorov:tauchen:2006}). Stable CARMA processes can be relevant in modelling energy markets (cf. \cite{Benth:Klu:Muller:Vos,Garciaetal2010}). Very often, a correct specification of the driving L\'evy process is of primary importance in all these applications.

In this paper we are concerned with a high-frequency sampled CARMA process driven by a second-order zero-mean L\'evy process. Under the assumption of \textit{invertibility} of the CARMA model, we present an $L^2$-{consistent} estimator for the increments of the driving L\'evy  process, employing standard time series techniques. It is remarkable that the proposed procedure works for arbitrary autoregressive and moving average orders, i.e. there is no need for \textit{order selection} in advance. In the light of the results in \cite{bfk:2011:2} and the {flexibility} of $\CARMA$ processes, the method might apply to a wider class of CMA models, too. Moreover, since the proof employes only the fact that the increments of the L\'evy process are orthogonal rather than independent, the result holds for a much broader class of driving processes. Notable examples are the $\COGARCH$ processes (\cite{BCL:2006,klm:2004}) or time-changed L\'evy processes (\cite{stovollevy}), which are often used to model volatility clustering in finance and  intermittency in turbulence. 

This noise recovery result gives rise to the conjecture that the sampled CARMA process behaves on a high-frequency time grid approximately like a suitable MA$(\infty)$ model that we call \textit{approximating Riemann sum process}. By comparing the asymptotic properties of the autocovariance structure of high-frequency sampled CARMA models with the one of their approximation Riemann sum processes, it will turn out that the so-called \textit{rule} of the Riemann sums plays a crucial role if the difference between the autoregressive and moving average order is greater than one. On the one hand, this gives new insight into the kernel estimation procedure studied in \cite{bfk:2011:2} and explains at which points the kernel is indeed estimated. On the other hand, this has obvious consequences for simulation purposes. Riemann sum approximations are an easy tool to simulate CMA processes. However, our results show that one has to be careful with the chosen rule of integration in the context of certain CARMA processes. 

The outline of the paper is as follows. In Section \ref{sec:preliminaries} we recall the definition of finite-variance CARMA models and summarise important properties of high-frequency sampled CARMA processes. In particular, we fix a global assumption that guarantees causality and invertibility for the sampled sequence. In the third section we then derive an $L^2$-consistent estimator for the increments of the driving L\'evy  process starting from the Wold representation of the sampled process. It will turn out that \textit{invertibility} of the original continuous-time model is sufficient and necessary for the recovery result to hold. Section \ref{sec:noise recovery} is completed by an illustrating example for CAR$(2)$ and CARMA$(2,1)$ processes. Thereafter, the high-frequency behaviour of approximating Riemann sum processes is studied in Section \ref{riemann:asympt}. First, an ARMA representation for the Riemann sum approximation is established in general and then the role of the rule of integration is analysed by matching the asymptotic autocovariance structure of sampled CARMA processes and their Riemann sum approximations in the cases where the autoregressive order is less or equal to three. The connection between the Wold representation and the approximating Riemann sum yields a deeper insight into the kernel estimation procedure introduced in \cite{bfk:2011:2}. 
The proof of Theorem~\ref{noise:extraction} and some auxiliary results can be found in the appendix. 
\section{{Preliminaries}} \label{sec:preliminaries}
\subsection{{Finite-variance CARMA processes}}
Throughout this paper we are concerned with a CARMA process driven by a second-order zero-mean L\'evy
process $L = \{L_t\}_{t\in\bbr}$ with $EL_1=0$ and $EL_1^2=1$. It is defined as follows.

For non-negative integers $p$ and $q$ such that $q<p$, a $\CARMA(p,q)$ process $Y=\{Y_t\}_{{t\in\bbr}}$ with
real coefficients
$a_{1},\ldots,a_{p}$, $b_{0},\ldots,b_{q}$ and driving L\'evy process $L$ is defined to be a
strictly stationary solution of the
suitably interpreted formal equation
\begin{equation}\label{1.1}
a(D)Y_t=\sigma b(D)DL_t,\quad t\in\bbr,
\end{equation}
where $D$ denotes differentiation with respect to $t$, $a(\cdot)$ and $b(\cdot)$ are the characteristic polynomials,
$$a(z):=z^{p}+a_{1}z^{p-1}+\cdots+a_{p}\quad\mbox{and}\quad
b(z):=b_{0}+b_{1}z+\cdots+b_{p-1}z^{p-1},$$
the
coefficients $b_{j}$ satisfy $b_{q}=1$ and $b_{j}=0$ for $q<j<p$, and $\sigma$ is a positive constant.
The polynomials $a(\cdot)$ and $b(\cdot)$ are assumed to have no common zeroes.
We denote, respectively, by $\lambda_i$ 
and $-\mu_i$ 
the roots of $a(\cdot)$ and $b(\cdot)$, such that these polynomials can be written as $a(z)=\prod_{i=1}^p(z-\lambda_i)$ and $b(z)=\prod_{i=1}^q(z+\mu_i)$.
Moreover, we suppose permanently 
\begin{Assumption} \label{assumption:causal}\begin{itemize}
\item[(i)]
 The zeroes of the polynomial $a(\cdot)$ satisfy $\Re(\la_j)<0$ for every $j=1,\ldots,p$,
\item[(ii)] and the roots of $b(\cdot)$ have non-vanishing real part, i.e. $\Re(\mu_j)\neq 0$ for all $j=1,\ldots,q$.
\end{itemize}
\end{Assumption}
\noindent

Since the derivative $DL_t$ does not exist in the usual
sense, we interpret~(\ref{1.1}) as being equivalent to the
observation and state equations
\begin{equation}\label{1.2}
Y_t=\mathbf{b}^{T}\mathbf{X}_t\,,
\end{equation}
\begin{equation}\label{1.3}
d\mathbf{X}_t=A\mathbf{X}_tdt+\mathbf{e}_{p}dL_t\,,
\end{equation}
where $$\mathbf{X}_t=\left(
          \begin{array}{c}
            X(t) \\
            X^{(1)}(t) \\
            \vdots \\
            X^{(p-2)}(t) \\
            X^{(p-1)}(t) \\
          \end{array}
        \right),\quad
       \mathbf{b}=\left(
       \begin{array}{c}
         b_{0} \\
         b_{1} \\
         \vdots \\
         b_{p-2} \\
         b_{p-1} \\
       \end{array}
     \right),\quad
     \mathbf{e}_{p}=\left(
       \begin{array}{c}
         0 \\
         0 \\
         \vdots \\
         0 \\
         1 \\
       \end{array}
     \right),$$
     $$\hskip .9in A=\left(
        \begin{array}{ccccc}
          0 & 1 & 0 & \ldots & 0 \\
          0 & 0 & 1 & \ldots & 0 \\
          \vdots & \vdots & \vdots &   \ddots & \vdots \\
          0 & 0 & 0 & \ldots & 1 \\
         -a_{p} & -a_{p-1} & -a_{p-2} & \ldots & -a_{1} \\
        \end{array}
      \right)\quad \text{and $A=-a_{1}$ for $p=1$}.$$
It is easy to check that the eigenvalues of the matrix $A$ are the same as the zeroes of the autoregressive polynomial $a(\cdot)$.

Under Assumption \ref{assumption:causal}(i) it has been shown {in} (\cite{BrLi}, Theorem~3.3) that Eqs. \eqref{1.2}-\eqref{1.3} have the unique strictly stationary solution
\begin{equation}\label{1.4}
Y_t=\int_{-\infty}^\infty g(t-u)dL_u,\quad t\in\bbr,
\end{equation}
where
\begin{equation} \label{jordan}g(t)=\begin{cases}{\displayfrac{\sigma}{2\pi i}}\displaystyle\int_{\rho}{\displayfrac{b(z)}{a(z)}e^{tz}dz}=\sigma\dsum_\lambda Res_{z=\lambda}\left(e^{zt}\displayfrac{b(z)}{a(z)}\right), ~&{\rm if}~t>0,\cr
                                                             0, ~&{\rm if}~ t\le 0,\end{cases}
\end{equation}
and $\rho$ is any simple closed curve in the open left half of the complex plane encircling the zeroes of $a(\cdot)$.  The sum  is over the distinct zeroes $\lambda$ of $a(\cdot)$ and $Res_{z=\lambda}(\cdot)$ denotes the residue at $\lambda$ of the function in brackets.
The kernel $g$ can be expressed (cf.~\cite{BrLi}, Equations (2.10) and (3.7)) also as
\begin{equation}\label{jordan:0}
g(t)=\sigma
\mathbf{b^\top} e^{At}\mathbf{e}_{p}{\bf 1}_{(0,\infty)}(t),\quad t\in\bbr,
\end{equation}
and its Fourier transform is
\begin{equation} \label{fourier}
\mathcal{F}\left\{g(\cdot)\right\}(\omega):=\int_{\bbr}g(s)e^{i\omega s}ds=\sigma\frac{b(-i\omega)}{a(-i\omega)},\quad \omega\in\bbr.
\end{equation}
In the light of Eqs.~\eqref{1.4}-\eqref{fourier}, we can interpret a $\CARMA$ process as a continuous-time filtered white noise whose transfer function has a finite number of poles and zeroes.
We emphasise that the condition on the roots of $a(\cdot)$ to lie in the interior of the left half of the complex plane in order to have causality arises from Theorem V, p. 8, \cite{Paley-Wiener}, which is intrinsically connected with the theorems in \cite{Titchmarsh}, pp. 125-129, on the Hilbert transform. A similar request on the roots of $b(\cdot)$ will turn out to be necessary for recovering the driving L\'evy process.

%
\subsection{{Properties of high-frequency sampled CARMA processes}}
We now recall some properties of the sampled sequence {$Y^\Delta:=\{Y_{n\Delta}\}_{n\in\bbz}$} of a {$\CARMA(p,q)$ process where $\Delta>0$}; cf. \cite{bfk:2011:2,bfk:2011:1} and references therein.
It is known that the sampled process {$Y^\Delta$} 
satisfies the $\ARMA(p,p-1)$ equations 
\begin{equation}
\Phi_\Delta(B)Y^\Delta_n=\Theta_{\Delta}(B)Z^{\Delta}_n,\quad n \in\bbz,~~\{Z_n^\Delta\}\sim{\rm WN}(0,\sigma^2_\Delta), \label{sampled:carma}
\end{equation}
with the $\AR$ part $\Phi_\Delta(B):=\prod_{i=1}^p(1-e^{\Delta\lambda_i}B)$, where $B$ is the discrete-time backshift operator, $BY_{n}^\Delta := Y_{n-1}^\Delta$.
Finally, the $\MA$ part $\Theta_\Delta(\cdot)$ is a polynomial of order $p-1$, chosen in such a way that it has no roots inside the unit circle. For {$p>3$ and} fixed $\Delta\g 0$ there is {no explicit expression for the} 
coefficients of $\Theta_\Delta(\cdot)$ nor the white noise process $Z^\Delta$.
Nonetheless, asymptotic expressions for $\Theta_\Delta(\cdot)$ and $\sigma^2_\Delta=\var(Z^\Delta_n)$ as $\Delta\downarrow 0$ were obtained in  \cite{bfk:2011:2,bfk:2011:1}. Namely, we have that the polynomial $\Theta_\Delta(z)$ and the variance $\sigma^2_\Delta$ can be written as (see Theorem 2.1, \cite{bfk:2011:2})
\begin{align}
\Theta_{\Delta}(z)=\prod_{i=1}^{p-q-1}(1+\eta(\xi_i)z)\prod_{k=1}^q(1-\zeta_k z),\quad z\in\bbc,\label{MA:factors}\\
\sigma^2_\Delta=	\frac{\sigma^2\Delta^{2(p-q)-1}}{(2(p-q)-1)!\prod_{i=1}^{p-q-1}\eta(\xi_i)}(1+o(1))\quad\textrm{as } {\Delta\downarrow 0},\label{prediction:error}
\end{align}
where, {again as $\Delta\downarrow 0$}, 
\begin{align}
\zeta_k&=1\pm\mu_k\Delta+o(\Delta),& k=1,\ldots,q,\notag\\
\eta(\xi_i)&=\xi_i-1\pm\sqrt{(\xi_i-1)^2-1}+o(1),& i=1,\ldots,p-q-1.\label{geometric:roots}
\end{align}
{The signs} $\pm$ {in \eqref{geometric:roots} are} chosen in such a way that, {for sufficiently small $\Delta$}, the coefficients {$\zeta_k$ and $\eta(\xi_i)$} are {less than} 
one in absolute value. This ensures that Eq.~\eqref{sampled:carma} is invertible. {Moreover}, $\xi_i$ 
are the zeroes of the function $\alpha_{p-q-1}(\cdot)$ that is defined as the {$(p-q-1)$-th} coefficient {in} the series expansion
\begin{equation}\label{expansion:geometric:roots}
\frac{\sinh(z)}{\cosh(z)-1+x}=\sum_{k=0}^{\infty}\alpha_k(x)z^{{2k+1}},\quad z\in{\bbc}, \ x\in\bbr\backslash\{0\},
\end{equation}
where the LHS of Eq.~\eqref{expansion:geometric:roots} is a power transfer function arising from the sampling procedure (cf. \cite{bfk:2011:1}, Eq. (11)). Therefore the coefficients $\eta(\xi_i)$ 
can be regarded as spurious since they do not depend on the parameters of the underlying continuous-time process $Y$, but just on $p-q$. 
\begin{oss}\label{remarkSpecFact}
Our notion of sampled process is a weak one since we require only that the sampled sequence has the same autocovariance structure as the continuous-time model observed on a discrete grid. We know that the filtered process on the LHS of \eqref{sampled:carma} (\cite{BrLi}, Lemma 2.1) is a $(p-1)$-dependent discrete-time process. Therefore there exist $2^{p-1}$ possible representations for the RHS of \eqref{sampled:carma}, each yielding the same autocovariance function of the filtered process, but only one has its roots outside the unit circle. The latter is called minimum-phase spectral factor (see \cite{Sayed} for a review on the topic). Since it is not possible to discriminate between the different factorisations, we always take the minimum-phase spectral factor without any further question. This will be crucial for our main result.

Moreover, the rationale behind Assumption \ref{assumption:causal}(ii) becomes clear now: if $\Re(\mu_k)=0$ for some $k$, then the corresponding $|\zeta_k|^2$ is equal to $1+\Delta^2|\mu_k|^2 + o(\Delta^2)$ for either sign choice. In this case, the $\MA(p-1)$ polynomial in Eq. \eqref{MA:factors} cannot be invertible for small $\Delta$.

\end{oss}
To ensure that the sampled $\CARMA$ process is invertible, we need to verify that $|\eta(\xi_i)|$ is strictly less than one for 
sufficiently small $\Delta$.
\begin{prop}\label{rootsgreater1}The {coefficients} $\eta(\xi_i)$ in Eq.~\eqref{geometric:roots} are uniquely determined by
$$\eta(\xi_i)=\xi_{i}-1-\sqrt{(\xi_i-1)^2-1} + o(1),\quad i=1,\ldots, p-q-1,$$
 and we have that $\xi_{i}-1-\sqrt{(\xi_i-1)^2-1}\in(0,1)$ for all $i$. 
\end{prop}
\begin{proof} 
{It follows from Proposition \ref{prop:real:roots} that $\xi_i\in(2,\infty)$ for all $i=1,\ldots,p-q-1$. This yields $\xi_{i}-1+\sqrt{(\xi_i-1)^2-1}>1$ for all $i$ and hence, we have that 
$$ \eta(\xi_i)=\xi_{i}-1-\sqrt{(\xi_i-1)^2-1} + o(1),\quad i=1,\ldots, p-q-1. $$
Since the {first-order} term of $\eta(\xi_i)$ is positive and monotonously decreasing in $\xi_i$, the additional claim follows.}
\end{proof}

\section{Noise recovery} \label{sec:noise recovery}
In this section we prove the first main statement {of the paper}, a recovery result for the driving L\'evy process. 
We start with some motivation for our approach.

We know that {the sampled CARMA sequence $Y^\Delta=\{Y_{n\Delta}\}_{n\in\bbz}$  
has the Wold representation (cf. \cite{BD}, {p. 187})
\begin{equation}\label{Wold}Y^\Delta_n=\sum_{j=0}^\infty \psi_j^\Delta Z_{n-j}^\Delta=\sum_{j=0}^\infty \left(\frac{\sigma_\Delta}{\sqrt{\Delta}}\psi_j^\Delta\right)\left(\frac{\sqrt{\Delta}}{\sigma_\Delta} Z_{n-j}^\Delta\right),\quad n\in\bbz,\end{equation}
where 
$\sum_{j=0}^\infty(\psi^\Delta_j)^2<\infty$.
Moreover, 
Eq.~\eqref{Wold} is the causal representation of Eq.~\eqref{sampled:carma}, and it has been shown in \cite{bfk:2011:2} that for every causal and invertible $\CARMA(p,q)$ process, as $\Delta\downarrow 0$,
\begin{equation}\frac{\sigma_\Delta}{\sqrt{\Delta}}\psi_{\lfloor t/\Delta\rfloor}^\Delta\rightarrow g(t),\quad t\geq 0,\label{convergence:kernel}\end{equation}
where $g$ is the kernel {in the moving average} representation \eqref{1.4}. 
Given the availability of classical time series methods to estimate $\{\psi^\Delta_j\}_{j\in\bbn}$ and $\sigma^2_\Delta$, and the flexibility of $\CARMA$ {processes}, we argue that {this result} can be applied to {more general continuous-time moving average} models. 

In view of {Eqs.~\eqref{Wold} and \eqref{convergence:kernel}} 
it is natural to investigate whether the quantity
$$\bar{L}^\Delta_{n}:=\frac{\sqrt{\Delta}}{\sigma_\Delta}Z^\Delta_n,\quad n\in\bbz,$$
approximates the increments of the driving L\'evy process in the sense that for every fixed $t>0$,
\begin{equation}\sum_{i=1}^{\lfloor t/\Delta\rfloor}{\bar{L}^\Delta_{i}}\overset{L^2}{\rightarrow} L_{t}\quad\text{as }\Delta \downarrow 0. \label{claim}\end{equation}

The first results on retrieving the increments of $L$ were given in \cite{bdy:2010}, and further generalized to the multivariate case by \cite{Eckard-noise}. 
The essential limitation of {this parametric method is that it might not be robust with respect to model misspecification. 
More precisely, the fact that a $\CARMA(p,q)$ process is $(p-q-1)$-times differentiable (see Proposition 3.32 of \cite{MaSt:2007}) is crucial for the procedure to work (cf. Theorem 4.3 of \cite{Eckard-noise}). However, if
the underlying process is instead $\CARMA(p',q')$ with $p'-q'\l p-q$, then some of the necessary derivatives do not exist anymore.}
In contrast to the aforementioned procedure, in the method we propose there is no need to specify the autoregressive and the moving average orders $p$ and $q$ in advance.


Before we start to prove the recovery result in Eq.~\eqref{claim}, let us establish the notion of invertibility in analogy to the discrete-time case.
\begin{defn}\label{def:invertible} A $\CARMA(p,q)$ process is said to be invertible if the roots of the moving average polynomial $b(\cdot)$ have  negative real parts, i.e. $\Re(\mu_i)>0$ for all $i=1,\dots,q$.
\end{defn}
\noindent
Our main theorem is the following. Its proof can be found in the appendix.
\begin{thm}\label{noise:extraction}
Let $Y$ be a finite-variance $\CARMA(p,q)$ process and $Z^\Delta$ the noise on the RHS of the sampled Eq.~\eqref{sampled:carma}. {Moreover, let Assumption \ref{assumption:causal}} hold and define $\bar{L}^\Delta:=\sqrt{\Delta}/\sigma_\Delta Z^\Delta$. Then, as $\Delta\downarrow 0$, 
\begin{equation}\label{claim2}
\sum_{i=1}^{\lfloor t/\Delta\rfloor}\bar{L}^\Delta_i\overset{L^2}{\rightarrow} L_t,\quad t\in(0,\infty),
\end{equation} if {and only if} the roots of the moving average polynomial $b(\cdot)$ on the RHS of the $\CARMA$ Eq.~\eqref{1.1} have negative real parts, i.e. if {and only if} the $\CARMA$ process is invertible.
\end{thm}
\begin{oss}
\begin{itemize}
\item[(i)] It is an easy consequence of the triangle and H\"older's inequality that, if the recovery result \eqref{claim2} holds, then also 
$$ \sum_{i=1}^{\lfloor t/\Delta\rfloor}\bar{L}^\Delta_i\sum_{j=\lfloor t/\Delta\rfloor+1}^{\lfloor s/\Delta\rfloor}\bar{L}^\Delta_j\overset{L^1}{\rightarrow} L_t(L_s - L_t),\quad t,s\in(0,\infty),\ t\leq s, $$ 
is valid.
\item[(ii)] Minor modifications of the proof of Theorem~\ref{noise:extraction} show that the recovery result in Eq.~\eqref{claim2} remains still valid if we drop the assumption of causality, Assumption~\ref{assumption:causal}(i), and suppose instead only $\Re(\la_j)\neq 0$ for every $j$. Hence, invertibility of the CARMA process is necessary for the noise recovery result to hold, whereas causality is not. Note that the white noise process in the non-causal case is not the same as in the Wold representation~\eqref{Wold}.
\item[(iii)] The necessity and sufficiency of the invertibility assumption descends directly from the fact that we choose always the minimum-phase spectral factor as pointed out in Remark~\ref{remarkSpecFact}. 
\item[(iv)] The proof of Theorem~\ref{noise:extraction} suggests that this procedure should work in a much more general framework. Let $I^\Delta(\cdot)$ denote the inversion filter in Eq.~\eqref{interpolated:noise} and $\psi^\Delta:=\left\{\psi^\Delta_i\right\}_{i\in\bbn}$ the coefficients in the Wold representation \eqref{Wold}. The proof essentially needs, apart from the rather technical Lemma \ref{dominant:integrand}, that, as $\Delta\downarrow 0$,
\begin{equation}
I^\Delta(e^{i\omega\Delta})\mathcal{F}\{g(\cdot)\}(\omega)=\frac{\int_{0}^\infty g(s)e^{i\omega s}ds}{\sum_{k=0}^\infty \psi^{\Delta}_ke^{i k \omega\Delta}}\rightarrow 1, \quad \omega \in\bbr,\label{invertibility}
\end{equation}
provided that the function $\sum_{k=0}^\infty\psi^\Delta_k z^k$ does not have any zero inside the unit circle. In other words, we need that the discrete Fourier transform in the denominator of Eq. \eqref{invertibility} converges to the Fourier transform in the numerator; this can be easily related to the kernel estimation result in Eq. \eqref{convergence:kernel}. Given the peculiar structure of $\CARMA$ processes, this relationship can be calculated explicitly, but the results should hold true for continuous-time moving average models with more general kernels, too.

\end{itemize}
%
\end{oss}
We illustrate Theorem \ref{noise:extraction} and the necessity of the invertibility assumption by an example where the convergence result is established using a time domain approach. That gives an explicit result also when the invertibility assumption is violated. 

Unfortunately this strategy is not  viable for a general $\CARMA$ process due to the complexity of involved calculations when $p$ is greater than two.
\begin{example}[$\CARMA(2,q)$ process]
The causal $\CARMA(2,q)$ process is the strictly stationary solution to the formal stochastic differential {equation}
\begin{align*}
(D-\lambda_2)(D-\lambda_1)Y_t&=\sigma DL_t,&q=0,\\
(D-\lambda_2)(D-\lambda_1)Y_t&=\sigma(b+D)DL_t,&q=1,
\end{align*}
where $\lambda_1,\lambda_2<0$, $\lambda_1\neq\lambda_2$ and $b\in\bbr\backslash\{0\}$.
It can be represented as a continuous-time moving average process as in Eq.~\eqref{1.4}, with kernel function
\begin{align*}
g(t)=&\sigma\frac{e^{t \lambda _1}-e^{t \lambda _2}}{\lambda _1-\lambda _2},&q=0,\\
g(t)=&\sigma\frac{b+\lambda _1}{\lambda _1-\lambda _2}e^{t \lambda _1} +\sigma\frac{b+\lambda _2}{\lambda _2-\lambda _1}e^{t \lambda _2},&q=1,
 \end{align*}
 {for $t>0$} and $0$ elsewhere.
The corresponding sampled process $Y^\Delta_n=Y_{n\Delta}$, $n\in\bbz$, satisfies the causal and invertible $\ARMA(2,1)$ equations as in \eqref{sampled:carma}.
From Eq. (27) of \cite{bfk:2011:1} we know for any $n\in\bbz$ that
\begin{align*} 
\Phi_\Delta(B)Y^\Delta_n&=\int_{{(n-1)\Delta}}^{n\Delta}{g({n\Delta}-u)dL_u}+\int_{{(n-2)\Delta}}^{(n-1)\Delta}[g(n\Delta-u)-(e^{\lambda_1\Delta}+e^{\lambda_2\Delta})g((n-1)\Delta-u)]dL_u.\end{align*}
The corresponding $\MA(1)$ polynomial in Eq. \eqref{sampled:carma} is $\Theta_\Delta(B)=1-\theta_\Delta B$, with asymptotic parameters
\begin{align*}
\theta_\Delta=\sqrt{3}-2+o(1),&\quad\sigma_\Delta^2=\sigma^2\Delta^3(2+\sqrt{3})/6+o(\Delta^3),&q=0,\\
\theta_\Delta=1-{\rm sgn}(b)\,b\,\Delta+o(\Delta),&\quad \sigma^2_\Delta=\sigma^2\Delta+o(\Delta),&q=1.
\end{align*}
Inversion of Eq.~\eqref{sampled:carma} 
gives, for every $\Delta>0$,
\begin{align}
Z^\Delta_n=&\frac{\Phi_\Delta(B)}{\Theta_\Delta(B)}Y^\Delta_n=\sum_{i=0}^\infty (\theta_\Delta B)^i\prod_{i=1}^2(1-e^{\lambda_i\Delta}B)Y^{\Delta}_n,
\notag\\
=&\int_{{(n-1)\Delta}}^{n\Delta}{g({n\Delta}-u)dL_u} \nonumber\\
&\quad+\sum_{i=0}^\infty\theta_\Delta^i\int_{{(n-i-2)\Delta}}^{(n-i-1)\Delta}[ g((n-i)\Delta-u)-(e^{\lambda_1\Delta}+e^{\lambda_2\Delta}-\theta_\Delta)g((n-i-1)\Delta-u)]dL_u.\notag
\end{align} 
The sequence $Z^\Delta:=\{Z^\Delta_n\}_{n\in\bbz}$ is a weak white noise process.
Moreover, using $\Delta L_n=\int^{n\Delta}_{(n-1)\Delta}dL_s$, we observe that 
\begin{equation}\label{crossvariance:noise}
\bbe[Z^\Delta_n\, \Delta L_{n-j}]=\left\{\begin{array}{lc}
0,&j<0,\\
\int_{0}^\Delta g(s)ds,&j=0,\\
\theta_\Delta^{j-1}\int_{0}^\Delta [g(\Delta+s)-(e^{\lambda_1\Delta}+e^{\lambda_2\Delta}-\theta_\Delta)g(s)]ds,&j\g 0.
\end{array}
\right.
\end{equation}
For any fixed {$t\in(0,\infty)$}, since $\Delta L$ and $\bar{L}^\Delta$ are both second-order stationary white noises with variance $\Delta$, we obtain that
\begin{align*}
\bbe\left[\sum_{i=1}^{\lfloor t/\Delta\rfloor}(\bar{L}^\Delta_i-\Delta L_i)\right]^2&=2\lfloor t/\Delta\rfloor\Delta-2\sum_{i=1}^{\lfloor t/\Delta\rfloor}\bbe[\bar{L}^\Delta_i\Delta L_i]-2\sum_{i\neq j}\bbe[\bar{L}^\Delta_i\Delta L_j]\\
&= 2\lfloor t/\Delta\rfloor\Delta-\frac{2\sqrt{\Delta}}{\sigma_\Delta}\lfloor t/\Delta\rfloor\int_{0}^\Delta g(s)ds \notag\\
&\qquad -\frac{2\sqrt{\Delta}}{\sigma_\Delta}\int_{0}^\Delta [g(\Delta+s)-(e^{\lambda_1\Delta}+e^{\lambda_2\Delta}-\theta_\Delta)g(s)]ds\sum_{i=1}^{\lfloor t/\Delta\rfloor}\sum_{j=1}^{i-1}\theta_\Delta^{j-1},
\end{align*}
where the last equality is deduced from Eq.~\eqref{crossvariance:noise}. For every {$a\neq 1$},
$$\sum_{i=1}^n\sum_{j=1}^{i-1}a^{j-1}=\frac{a^n+(1-a) n-1}{(1-a)^2},\quad n\in\bbn,$$
and the variance of the error can be explicitly calculated as
\begin{align*}\bbe&\left[\sum_{i=1}^{\lfloor t/\Delta\rfloor}(\bar{L}^\Delta_i-\Delta L_i)\right]^2=2\lfloor t/\Delta\rfloor\Delta-\frac{2\sqrt{\Delta}}{\sigma_\Delta}\lfloor t/\Delta\rfloor\int_{0}^\Delta g(s)ds \notag\\
&\quad-\frac{2\sqrt{\Delta}}{\sigma_\Delta}\frac{{\theta_\Delta}^{\lfloor t/\Delta\rfloor}+\lfloor t/\Delta\rfloor(1-\theta_\Delta )-1}{(1-\theta_\Delta)^2}\int_{0}^\Delta [g(\Delta+s)-(e^{\lambda_1\Delta}+e^{\lambda_2\Delta}-\theta_\Delta)g(s)]ds.
\end{align*}
We now compute the asymptotic  expansion for $\Delta\downarrow0$ of the equation above. We obviously have that $2\lfloor t/\Delta\rfloor\Delta=2t(1+o(1))$ and, using the explicit formulas for the kernel functions $g$,
{\footnotesize $$\begin{array}{r|l|l}
&\underline{q=0}&\underline{q=1}\\
\frac{2\sqrt{\Delta}}{\sigma_\Delta}\lfloor t/\Delta\rfloor\int_{0}^\Delta g(s)ds=&\left(3-\sqrt{3}\right)t+o(1),&2t+o(1),\\
\frac{2\sqrt{\Delta}}{\sigma_\Delta}\int_{0}^\Delta [g(\Delta+s)-(e^{\lambda_1\Delta}+e^{\lambda_2\Delta}-\theta_\Delta)g(s)]ds=&\left(4 \sqrt{3}-6\right) \Delta(1+o(1)),&2 (b-{{\rm sgn}(b)}\,b) \Delta ^2+o(\Delta ^2),\\
({\theta_\Delta}^{\lfloor t/\Delta\rfloor}+\lfloor t/\Delta\rfloor(1-\theta_\Delta )-1){(1-\theta_\Delta)^{-2}}=&\frac{1}{6} \left(3+\sqrt{3}\right)t/\Delta(1+o(1)),&(e^{-{{\rm sgn}(b)}b t}+{{\rm sgn}(b)} bt-1)/( b \Delta)^2+o(\Delta^{-2}).
\end{array}$$}
{Hence}, for a fixed {$t\in(0,\infty)$} and $\Delta\downarrow0$, we {get}
$$\bbe\left[\sum_{i=1}^{\lfloor t/\Delta\rfloor}(\bar{L}^\Delta_i-\Delta L_i)\right]^2=
\left\{\begin{array}{cc}
o(1),&q=0,\\
2 (e^{-{{\rm sgn}(b)}b t}+{{\rm sgn}(b)}bt-1) {({\rm sgn}(b)-1)/b} + o(1),&q=1,
\end{array}\right.
$$
i.e. \eqref{claim2} holds always for $q=0$, {whereas} for $q=1$ if and only if {$b>0$}. If {$b<0$}, the error made by approximating the driving L\'{e}vy by inversion of the discretised process grows as  $4t$ for large $t$. \end{example}
\section{High-frequency behaviour of approximating {Riemann} sums}\label{riemann:asympt}
The fact that, in the sense of Eq.~\eqref{claim}, $\bar{L}_n^\Delta\approx \Delta L_n = L_{n\Delta}-L_{(n-1)\Delta}$ for small $\Delta$, along with Eq.~\eqref{convergence:kernel}, gives {rise to the conjecture} 
that the  Wold representation 
for ${Y}^{\Delta}$ behaves {on a high-frequency time grid} approximately like 
the $\MA(\infty)$ process
\begin{equation}\label{Riemann:approx}\tilde{Y}^{\Delta,h}_n:=\sum_{j=0}^\infty g(\Delta(j+h))\Delta L_{n-j},\quad n\in \bbz,\end{equation}
with some $h\in[0,1]$ and {$g$ is the kernel function as in \eqref{jordan:0}}. 
{In other terms}, we have for a  $\CARMA$ process, under the assumption of invertibility and causality, that the discrete-time quantities appearing in the Wold representation 
approximate the quantities in Eq.~\eqref{Riemann:approx} when $\Delta\downarrow 0$.
The transfer function of Eq.~\eqref{Riemann:approx} is defined as 
\begin{equation}\psi^{\Delta}_h(\omega):=\sum_{j=0}^\infty g(\Delta(j+h))e^{-i\omega j},\quad -\pi\leq\omega\leq\pi,\label{Riemann}\end{equation}
 and its spectral density can be written as 
$$\tilde{f}^\Delta_h(\omega)=\frac{{1}}{2\pi}|\psi^{\Delta}_h|^2(\omega),\quad   -\pi\leq\omega\leq\pi.$$
{It is well known} that a {CMA process} can be defined {(for a fixed time point $t$)} as the $L^2$-limit of Eq.~\eqref{Riemann:approx}; this fact is naturally employed to simulate a CMA model when all the relevant quantities are known a priori.
Therefore, we call $\tilde{Y}^{\Delta,h}$ \emph{approximating {Riemann} sum} of Eq.~\eqref{1.4}, and $h$ is said to be the \emph{rule} of the approximating sum. If, for instance, $h$ is chosen to be $1/2$, we have the popular  \emph{mid-point rule}.
%
\begin{oss} \label{other integration rules}
 \begin{enumerate}
  \item[(i)] It would be possible to consider more sophisticated integration rules by taking more nodes on every interval of length $\Delta$ and suitable weights. However, since mostly used in practice, we decided
   to concentrate on that ``simple'' Riemann sum approximation. 
  \item[(ii)] In practice, when considering simulation studies for instance, one has to use a finite (truncated) Riemann sum of the form
   $$\tilde{Y}_{n,\,N}^{\Delta,\,h} := \sum_{j=0}^N g(\Delta (j+h))\,\Delta L_{n-j},$$
   where $N\in\mathbb{N}$ is usually taken as a large number. If we let $N = N(\Delta) \to \infty$ as $\Delta \to 0$ with a suitable rate 
   ($N(\Delta)$ should diverge faster than $\Delta$ goes to $0$, e.g.~$N(\Delta) = \Delta^{-(1+\varepsilon)}$), the main result of this section, Corollary \ref{cor:sum}, remains valid.	
 \end{enumerate}
\end{oss}
To give an answer to our conjecture, we investigate 
{properties} of the approximating {Riemann} sum $\tilde{Y}^{\Delta,h}$ of a $\CARMA$ process {and compare its} asymptotic {autocovariance} structure 
with the one of the sampled CARMA sequence $Y^\Delta$. 
This yields more 
insight {into} the role of $h$ {for} the behaviour of $\tilde{Y}^{\Delta,h}$ as a process. 

We {start with a} 
well-known property of approximating sums.
\begin{prop}\label{L2convergence} 
Let $g$ be in  $L^2$ and Riemann-integrable. Then, for every $h\in[0,1]$, as $\Delta\downarrow 0$:
\begin{itemize} \item[(i)] $\tilde{Y}^{\Delta,h}_{k}-Y^\Delta_k\overset{L^2}{\rightarrow} 0$, for every $k\in\bbz$.
\item[(ii)] $\tilde{Y}^{\Delta,h}_{\lfloor t/\Delta\rfloor}\overset{L^2}{\rightarrow}Y_t$, for every $t\in\bbr$.
\end{itemize}
\end{prop}
\begin{proof}
{This follows immediately from the hypotheses made on $g$ and the definition of $L^2$-integrals.}
\end{proof}
\noindent
This result essentially says only that  approximating sums converge to $Y_t$ for every fixed time point $t$. However, for a $\CARMA(p,q)$ process we have that the {approximating Riemann sum process satisfies for every $h$ and $\Delta$ an $\ARMA(p,p-1)$ equation {(see Proposition \ref{prop:Riemann:arma} below)}. This means that there might exist a process whose autocorrelation structure is the same {as the one} of the {approximating sum}. 
Given 
{that the AR filter in this representation is the same as in Eq.~\eqref{sampled:carma}}, it is reasonable to investigate whether $\Phi_\Delta(B)Y^{\Delta}$ and $\Phi_\Delta(B)\tilde{Y}^{\Delta,h}$ have, as $\Delta\downarrow 0$, 
the same asymptotic autocovariance structure, which can be expected but is not granted by Proposition \ref{L2convergence}.

The following proposition states the ARMA($p, p-1$) representation for the approximating Riemann sum.
\begin{prop}\label{prop:Riemann:arma} Let $Y$ be a $\CARMA(p,q)$ process, satisfying Assumption \ref{assumption:causal}. Furthermore, suppose that the roots of the autoregressive polynomial $a(\cdot)$ are distinct. The approximating {Riemann} sum process $\tilde{Y}^{\Delta,h}$ of $Y$ defined by Eq.~\eqref{Riemann:approx} satisfies, for every $h\in[0,1]$, the $\ARMA(p,p-1)$ equation 
\begin{equation}\label{Riemann:arma}
\Phi_\Delta(B)\tilde Y^{\Delta,h}_n=\sigma\tilde{\Theta}_{\Delta,h}(B)\Delta L_n,\quad n\in\bbz,
\end{equation}
where 
\begin{equation}\tilde{\Theta}_{\Delta,h}(z):=\tilde{\theta}^{\Delta,h}_0-\tilde{\theta}^{\Delta,h}_{1}z+-\ldots+(-1)^{p-1}\tilde{\theta}^{\Delta,h}_{p-1}z^{p-1}\label{Riemann:MA}\end{equation}
and 
$$\tilde{\theta}^{\Delta,h}_k:={\sum_{l=1}^p \frac{b(\lambda_l)}{a'(\lambda_l)} e^{h \Delta  \lambda_l}\sum e^{\Delta(\lambda_{j_1}+\lambda_{j_2}+\ldots+\lambda_{j_{k}})}},\quad k=0,\ldots,p-1.$$
The right-hand sum is defined to be one for $k=0$ and it is evaluated over all possible subsets $\{j_1,\ldots,j_{k}\}$ of $\{1,\ldots,p\}\backslash\{l\}$ with cardinality $k$, if $k\g0$.
\end{prop}
\begin{proof}
Write $\Phi_\Delta(z) = \prod_{j=1}^p(1-e^{\Delta\la_j}z) = -\sum_{j=0}^p\phi_j^\Delta z^j$ and observe that 
\begin{eqnarray*}
 \lefteqn{\Phi_\Delta(B)\tilde Y^{\Delta,h}_n = -\sum_{j=0}^p\phi_j^\Delta Y^{\Delta,h}_{n-j}} \\
  &&= -\sigma\mathbf{b^\top}\sum_{k=0}^{p-1}\left(\sum_{j=0}^k\phi_j^\Delta e^{A(k-j)\Delta}\right)e^{Ah\Delta}\mathbf{e}_{p}\cdot\Delta L_{n-k} \\
  &&\qquad\qquad-\sigma\mathbf{b^\top}\sum_{j=0}^p\,\sum_{k=p-j}^{\infty}\phi_j^\Delta e^{A(h+k)\Delta}\mathbf{e}_{p}\cdot\Delta L_{n-j-k} \\
  &&= -\sigma\mathbf{b^\top}\sum_{k=0}^{p-1}\left(\sum_{j=0}^k\phi_j^\Delta e^{A(k-j)\Delta}\right)e^{Ah\Delta}\mathbf{e}_{p}\cdot\Delta L_{n-k} \\
  &&\qquad\qquad+\sigma\mathbf{b^\top}\sum_{k=p}^\infty\left(-\sum_{j=0}^{p} \phi_j^\Delta e^{-Aj\Delta}\right)e^{A(h+k)\Delta}\mathbf{e}_{p}\cdot\Delta L_{n-k}.
\end{eqnarray*}
By virtue of the Cayley-Hamilton Theorem (cf. also \cite[proof of Lemma 2.1]{BrLi}), we have that 
$$ -\sum_{j=0}^{p} \phi_j^\Delta e^{-Aj\Delta} = 0,$$
and hence, $\Phi_\Delta(B)\tilde Y^{\Delta,h}_n = -\sigma\mathbf{b^\top}\sum_{k=0}^{p-1}\left(\sum_{j=0}^k\phi_j^\Delta e^{A(k-j)\Delta}\right)e^{Ah\Delta}\mathbf{e}_{p}\cdot\Delta L_{n-k}$. We conclude with \cite[Lemma 2.1(i) and Eq.~(4.4)]{FasenFuchs2012}.
\end{proof}
\begin{oss}
\begin{enumerate}
\item[(i)]
The approximating Riemann sum of a causal $\CARMA$ process is automatically a causal $\ARMA$ process. On the other hand, even if the $\CARMA$ model is invertible in the sense of Definition \ref{def:invertible}, the roots of $\tilde{\Theta}_{\Delta,h}(\cdot)$ may lie inside the unit circle, causing $\tilde{Y}^{\Delta,h}$ to be non-invertible.
\item[(ii)]
It is easy to see that $\tilde{\theta}^{\Delta,h}_0=g(h\Delta)$. If {$p-q\geq 2$} and $h=0$, {we have that} $\tilde{\theta}^{\Delta,0}_0=0$, giving that $\tilde{\Theta}_{\Delta,0}(0)=0$. This is never the case for $\Theta_\Delta(\cdot)$ as one can see from Eq. \eqref{MA:factors} and Proposition \ref{rootsgreater1}.
Moreover, it is possible to show that for $h=1$ and {$p-q\geq 2$}, {the coefficient} $\tilde{\theta}^{\Delta,1}_{p-1}$ {is equal to $0$}, implying that \eqref{Riemann:arma} is actually an $\ARMA(p,p-2)$ equation. For those values of $h$, the $\ARMA$ equations solved by the approximating Riemann sums 
can never have the same asymptotic form {as} Eq.~\eqref{sampled:carma}. Therefore, we restrict ourselves to the case $h\in(0,1)$ from now on.
\item[(iii)] {The assumption of distinct autoregressive roots might seem restrictive, but the omitted cases can be obtained by letting distinct roots tend to each other. This would, of course, change the coefficients of the MA polynomial in Eq.~\eqref{Riemann:MA}.
Moreover, as shown in \cite{bfk:2011:2,bfk:2011:1}, the multiplicity of the zeroes does not matter when $L^2$-asymptotic relationships as $\Delta\downarrow0$ are considered.}
\end{enumerate}
\end{oss}

Due to the complexity of retrieving the roots of a polynomial of arbitrary order from its coefficients, 
{finding} the asymptotic expression of  $\tilde{\Theta}_{\Delta,h}(\cdot)$ for arbitrary $p$ 
is a daunting task. Nonetheless, by using Proposition \ref{prop:Riemann:arma}, it is not difficult to give an answer for processes with $p\leq 3$, which are the most used in practice. 
\begin{prop}\label{thm:Riemann:sum:h} Let $\tilde{Y}^{\Delta,h}$ be the approximating {Riemann} sum for a $\CARMA(p,q)$ process, suppose $p\leq3$, and {let Assumption \ref{assumption:causal} hold and  the roots of $a(\cdot)$ be distinct}.

If $p=1$, {the process} $\tilde{Y}^{\Delta,h}$ is {an} $\AR(1)$ process driven by $Z_n^{\Delta}={\sigma}e^{\Delta h\lambda_1}\Delta L_n$. 
If $p=2,3,$ we have
\begin{equation}\Phi_\Delta(B)\tilde{Y}^{\Delta,h}_n=\prod_{i=1}^{q}(1-(1-\Delta \mu_i +o(\Delta))B)\prod_{i=1}^{p-q-1}(1-\chi_{p-q,i}(h)B)\left(\sigma\frac{(h\Delta)^{p-q-1}}{(p-q-1)!}\Delta L_n\right),\label{MA:riemann:factors}\end{equation}
where, for $h\in(0,1)$ {and $\Delta\downarrow 0$}, 
\begin{align*}
 \chi_{2,1}(h) &=\frac{h-1}{h}+{o(1)}\quad\text{ and} \\
 \chi_{3,j}(h) &=\frac{2 (h-1)^2}{2 (h-1) h-1-(-1)^j\sqrt{1-4 (h-1) h}}+o(1),\quad j=1,2.
\end{align*}
\end{prop}
\begin{proof}
The polynomial $\tilde{\Theta}_{\Delta,h}(z)$ is of order $p-1$. 
Since $p\leq 3$, its roots, if any, can be calculated from the coefficients and asymptotic expressions can be obtained by computing the Taylor {expansions} of the roots around $\Delta=0$.

If $p=1$, the statement follows directly from Eq.~\eqref{Riemann:arma}. For $p=2,3$, the roots of Eq.~\eqref{Riemann:MA} are $\{1+\Delta\mu_i+o(\Delta)\}_{i=1,\ldots,q}$ and $\{1/\chi_{{p-q},i}(h)\}_{i=1,\ldots {p-q-1}}$, giving that
$$\tilde{\Theta}_{\Delta,h}(z)=\tilde{\theta}^{\Delta,h}_{p-1}\prod_{i=1}^{q}(1+\Delta \mu_i +o(\Delta)-z)\prod_{i=1}^{p-q-1}(1/\chi_{p-q,i}(h)-z),\quad z\in\bbc.$$
Vieta's Theorem shows that the product of the roots must be {equal to} $\tilde{\theta}^{\Delta,h}_0/\tilde{\theta}^{\Delta,h}_{p-1}$, {which yields}
$$\tilde{\Theta}_{\Delta,h}(z)=\tilde{\theta}^{\Delta,h}_0\prod_{i=1}^{q}(1-(1-\Delta \mu_i +o(\Delta))z)\prod_{i=1}^{p-q-1}(1-\chi_{p-q,i}(h)z).$$
Since $\tilde{\theta}^{\Delta,h}_0=g(h\Delta)= \sigma(h\Delta)^{p-q-1}/(p-q-1)!(1+o(1))$, we have established the result.
\end{proof}
In general, the autocorrelation structure depends on $h$ through the parameters $\chi_{p-q,i}(h)$. 
In a time series context, it is reasonable to require that the approximating Riemann sum has the same asymptotic autocorrelation structure as the $\CARMA$ process that we want to approximate.

\begin{cor}\label{cor:sum}
Let the assumptions of Proposition \ref{thm:Riemann:sum:h} hold. Then $\Phi_\Delta(B)Y^{\Delta}$ and $\Phi_\Delta(B)\tilde{Y}^{\Delta,h}$ have the same {asymptotic autocovariance structure} as $\Delta\downarrow0$   
\begin{align*}
\text{{for every} }h&\in(0,1),&\text{{if} }p-q=1,\\
\text{{for} }h&=(3{\pm}\sqrt{3})/6,&\text{{if} }p-q=2,\\
\text{{and for} }h&={\Big(15\pm\sqrt{225 - 30\sqrt{30}}\Big)/30},&\text{{if} }p-q=3.
\end{align*}
Moreover, the $\MA$ polynomials in Eqs. \eqref{MA:factors} and \eqref{MA:riemann:factors} coincide if and only if the $\CARMA$ process is invertible and $|\chi_{p-q,i}(h)|\l 1$, that is
\begin{align*}
\text{{for every} }h&\in(0,1),&\text{{if} }p-q=1,\\
\text{{for} }h&=(3+\sqrt{3})/6,&\text{{if} }p-q=2.
\end{align*}
For $p-q=3$, such an $h$ does not exist.
\end{cor}
\begin{proof}
%
%
{The claim for $p-q=1$ follows immediately from Proposition \ref{thm:Riemann:sum:h} and Eqs.~\eqref{MA:factors}-\eqref{prediction:error}. For $p=2$ and $q=0$, we have to solve the spectral factorization problem 
\begin{align*}
 \sigma_\Delta^2(1+\eta(\xi_1)^2) &= \sigma^2\Delta^3(1+\chi_{2,1}(h)^2) h^2 \\
 \sigma_\Delta^2\eta(\xi_1) &= -\sigma^2\Delta^3\chi_{2,1}(h)h^2 
\end{align*}
 with $\eta(\xi_1) = 2 - \sqrt{3} + o(1)$ and $\chi_{2,1}(h) = (h-1)/h + o(1)$. Equation~\eqref{prediction:error} then yields the two solutions $h = (3\pm\sqrt{3})/6$. For $p=3$ and $q=1$, analogous calculations lead to the same solutions. Finally, consider the case $p=3$ and $q=0$. We have to solve asymptotically the following system of equations  
\begin{align*}
 \sigma_\Delta^2(1+(\eta(\xi_1)+\eta(\xi_2))^2+\eta(\xi_1)^2\eta(\xi_2)^2) &= \frac{\sigma^2\Delta^5}{4}(1+(\chi_{3,1}(h)+\chi_{3,2}(h))^2 + \chi_{3,1}(h)^2\chi_{3,2}(h)^2) h^4 \\
 \sigma_\Delta^2(\eta(\xi_1)+\eta(\xi_2))(1+\eta(\xi_1)\eta(\xi_2)) &= -\frac{\sigma^2\Delta^5}{4}(\chi_{3,1}(h)+\chi_{3,2}(h))(1+\chi_{3,1}(h)\chi_{3,2}(h)) h^4 \\
 \sigma_\Delta^2\eta(\xi_1)\eta(\xi_2) &= \frac{\sigma^2\Delta^5}{4}\chi_{3,1}(h)\chi_{3,2}(h) h^4 
\end{align*}
 where $\eta(\xi_{1,2}) = \big(13\pm\sqrt{105}-\sqrt{270\pm26\sqrt{105}}\big)/2+o(1)$ 
 and $\chi_{3,1}(h)$ and $\chi_{3,2}(h)$ are as in Proposition \ref{thm:Riemann:sum:h}. Solving that system for $h$ gives the claimed values.}

To prove the second part of the corollary, we start observing that, under the assumption of an invertible $\CARMA$ process, the coefficients depending on $\mu_i$, if any, coincide automatically. Then it remains to check whether the coefficients depending on $h$ can be smaller than 1 in absolute value.
The cases $p-q=1,2$ follow immediately.
Moreover, to see that there is no such $h$ for $p-q=3$, it is enough to notice that, for any $h\in(0,1)$, we have $|\chi_{3,1}(h)|> 1$ and $0<|\chi_{3,2}(h)|< 1$. Hence, they never satisfy the sought requirement for $h\in (0,1)$.
\end{proof}
\begin{oss}
 It is also feasible to use spectral densities rather than covariances in the proof of Corollary \ref{cor:sum}. In that case, one has to compare the spectral densities of $\Phi_\Delta(B)Y^{\Delta}$ and $\Phi_\Delta(B)\tilde{Y}^{\Delta,h}$ asymptotically as $\Delta\to 0$. This would lead to the question whether the equation
\begin{equation} \label{eq:spectral densities}
 \sigma_\Delta^2\left|\Theta_{\Delta}(z)\right|^2 = \sigma^2\,\Delta\left|\tilde{\Theta}_{\Delta,h}(z)\right|^2
\end{equation}
holds for any $z\in\mathbb{C}$ with $|z|=1$ as $\Delta\to 0$. Of course, \eqref{eq:spectral densities} implies the same values for $h$ as those stated in Corollary \ref{cor:sum}.
\end{oss}

Corollary \ref{cor:sum} can be interpreted as a criterion to choose an $h$ such that the Riemann sum approximates the continuous-time process $Y$ in a stronger sense than the simple convergence as a random variable for every fixed time point $t$. 
The second part of the corollary says that there is an even more restrictive way to choose $h$ if we want Eqs.~\eqref{MA:factors} and \eqref{MA:riemann:factors} to coincide. If the two processes satisfy asymptotically the same causal and invertible ARMA equation, they have the same coefficients in their Wold representations as $\Delta\downarrow0$. In the case of the approximating Riemann sum these coefficients are given explicitly by definition in Eq.~\eqref{Riemann:approx}.

In the light of Eq. \eqref{convergence:kernel} and Theorem \ref{noise:extraction}, the sampled $\CARMA$ process behaves asymptotically like its 
approximating Riemann sum process for some specific $h=\bar{h}$, which might not even exist as in the case $p=3$, $q=0$. However, if such an $\bar{h}$ exists, the kernel estimators \eqref{convergence:kernel} can be improved to 
$$\frac{\sigma_\Delta}{\sqrt{\Delta}}\psi^\Delta_{\lfloor t/\Delta \rfloor}=  g(\Delta(\lfloor t/\Delta \rfloor+\bar{h}))+o(1),\quad t\in\bbr.$$ 

For invertible $\CARMA(p,q)$ processes with $p-q=1$, any choice of $h$ would accomplish that. In principle an $\bar{h}$ can be found by matching a higher-order expansion in $\Delta$, where higher-order terms depend on $h$.

For $p-q=2$, there is only a specific value $h=\bar{h}:=(3+\sqrt{3})/6$ such that $\tilde{Y}^{\Delta,\bar{h}}$ behaves as $Y^{\Delta}$ in this particular sense. Therefore, it advocates for a unique, optimal value for, e.g., simulation purposes.

Finally, for $p-q=3$, a similar value does not exist, meaning that it is not possible to mimic $Y^\Delta$ in this sense with any approximating Riemann sum. 

To confirm these observations, we now give a small numerical study.
We consider three different causal and invertible processes, a $\CARMA(2,1)$, a $\CAR(2)$, and a $\CAR(3)$ model with parameters $\lambda_1=-0.7$, $\lambda_2=-1.2$, $\lambda_3=-2.6$ and $\mu_1=3$. Of course, for the $\CARMA(2,1)$ we use only $\lambda_1,\lambda_2$ and $\mu_1$, whereas for the $\CAR$ processes there is no need for $\mu_1$. We estimate the kernel functions from the theoretical autocorrelation functions using \eqref{convergence:kernel} as in \cite{bfk:2011:2}. Our sampling rates are moderately high, namely $2^2=4$ (Figure \ref{kern22}) and $2^6=64$ samplings per unit of time (Figure \ref{kern26}). To see where the kernel is being estimated, we plot the kernel estimations on different grids. The small circles denote the extremal cases $h=0$ and $h=1$, the vertical sign the mid-point rule $h=0.5$, and the diamond and the square are the values given in Corollary \ref{cor:sum}, if any. The true kernel function is then plotted with a solid, continuous line. For the sake of clarity, only the first eight estimates are plotted. 

For the $\CARMA(2,1)$ process, the kernel estimation seems to follow a mid-point rule (i.e. $h=1/2$). For the $\CAR(2)$ process, the predicted value $\bar{h}=(3+\sqrt{3})/6$ (denoted with squares) is definitely the correct one, and for the $\CAR(3)$ the estimation is close for every $h\in [0,1]$, but constantly biased. In the limit $\Delta\downarrow 0$, the slightly weaker results given by Eq.~\eqref{convergence:kernel} still hold, showing that the bias vanishes in the limit. 
The conclusion expressed above is true for both considered sampling rates, which is remarkable since they are only moderately high in comparison with the chosen parameters.

\begin{figure}[h!]
\includegraphics[width=\columnwidth]{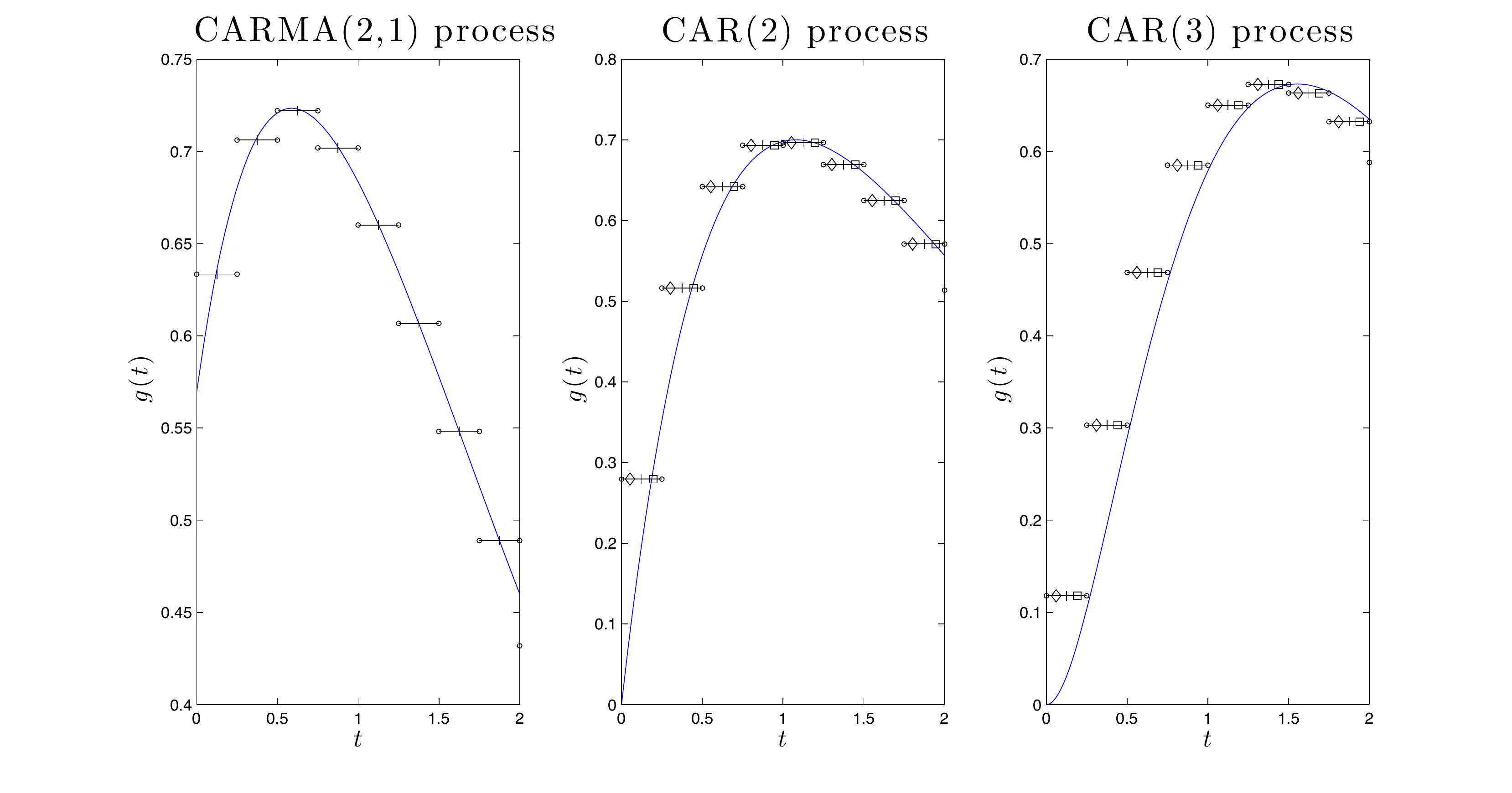}
\caption{Kernel estimation for a sampling frequency of $2^2$ samplings per unit of time, i.e. $\Delta=0.25$. The diamond and the square symbols denote, if available, the values of $h$ suggested by Corollary \ref{cor:sum}.}
\label{kern22}
\end{figure}
\begin{figure}[h!]
\includegraphics[width=\columnwidth]{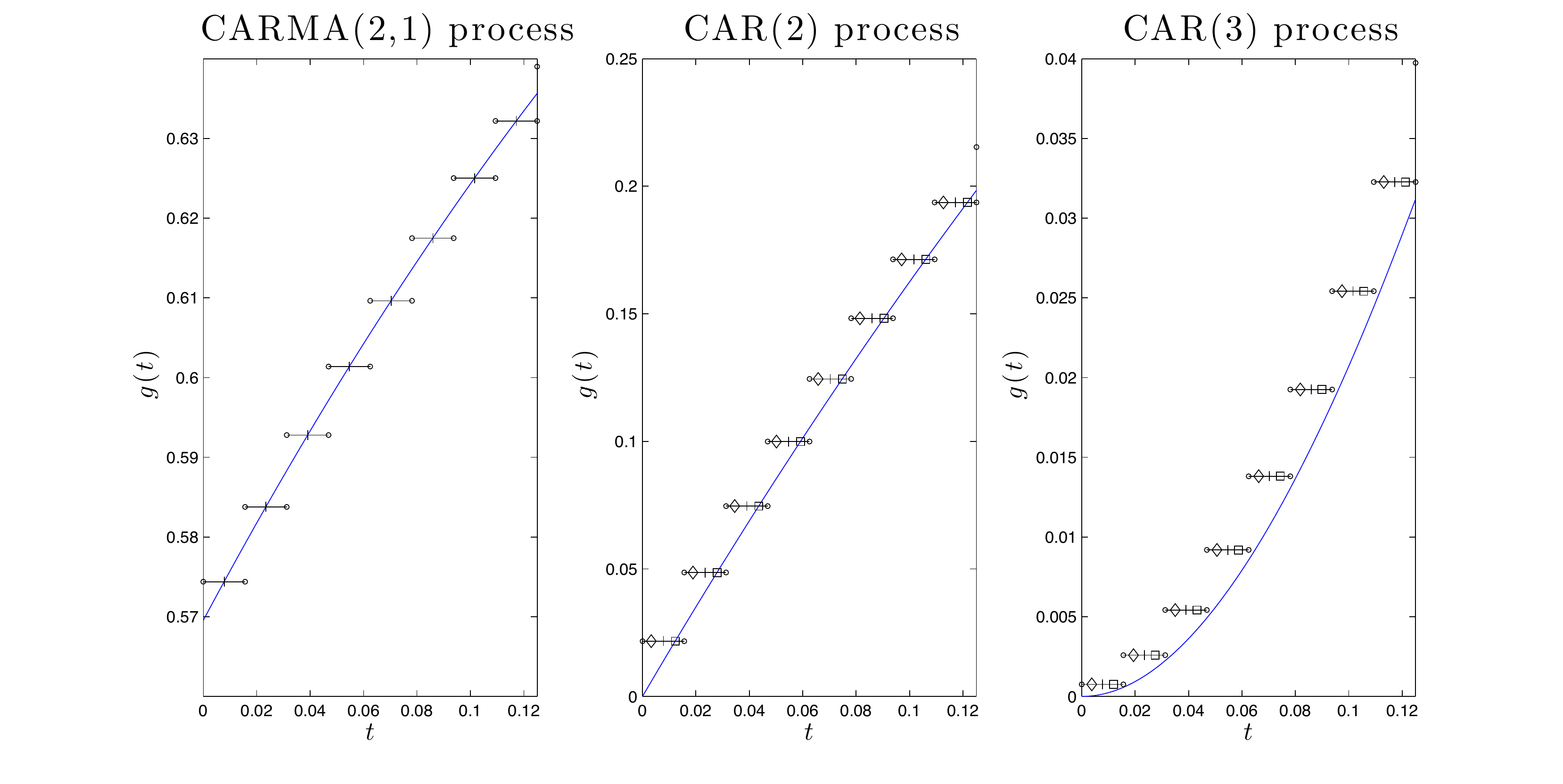}
\caption{Kernel estimation for a sampling frequency of $2^6$ samplings per unit of time, i.e. $\Delta\approx 0.016$. The diamond and the square symbols denote, if available, the values of $h$ suggested by Corollary \ref{cor:sum}.}
\label{kern26}
\end{figure}

\begin{appendix}
 
\section{Proof of Theorem~\ref{noise:extraction} and auxiliary results}
Throughout the appendix, we use the same notation as in the preceding sections. We start with the proof of our main theorem in Section~\ref{sec:noise recovery}.
\begin{proof}[\textit{Proof of Theorem~\ref{noise:extraction}.}]
Due to Assumption \ref{assumption:causal}(ii) and Proposition \ref{rootsgreater1}, the sampled $\ARMA$ equation \eqref{sampled:carma} is invertible.
The noise on the RHS of Eq.~\eqref{sampled:carma} is then obtained using the classical inversion formula 
$$Z^\Delta_n=\frac{\Phi_\Delta(B)}{\Theta_\Delta(B)}Y^\Delta_{n},\quad n\in\bbz,$$
where $B$ is the usual backshift operator.
Let us consider the stationary \emph{continuous-time} process 
\begin{equation}\mathcal{Z}^\Delta_{t}:=\frac{\Phi_\Delta(B_\Delta)}{\Theta_\Delta(B_\Delta)}Y_t=\sum_{i=0}^\infty a_i^\Delta\int_{-\infty}^{t-i\Delta}g(t-i\Delta-s)dL_s,\quad t\in\bbr,\label{interpolated:noise}\end{equation}
where the continuous-time backshift operator $B_\Delta$ is defined such that $B_\Delta Y_t:=Y_{t-\Delta}$ for every $t\in\bbr$. The coefficients $a_i^\Delta$ on the RHS of Eq.~\eqref{interpolated:noise} are determined by the Laurent series expansion of the rational function $\Phi_\Delta(\cdot)\Theta^{-1}_\Delta(\cdot)$. Moreover, $\mathcal{Z}^\Delta_{n\Delta}=Z^\Delta_{n}$ for every $n\in\bbn$; as a consequence, the random variables $\mathcal{Z}^\Delta_s,\mathcal{Z}^\Delta_t$ are uncorrelated for $|t-s|\geq \Delta$ and $\var(\mathcal{Z}^\Delta_t)=\var(Z^\Delta_n)$.
Exchanging the sum and the integral signs in Eq. \eqref{interpolated:noise}, and since $g(\cdot)=0$ for negative arguments, we have that $\mathcal{Z}^\Delta$ is a continuous-time moving average process 
$$\mathcal{Z}^\Delta_{t}=\int_{-\infty}^{t} {g}^\Delta(t-s)dL_s,\quad t\in\bbr,$$
whose kernel function $g^\Delta$ has Fourier transform (cf. Eq.~\eqref{fourier})
$$\mathcal{F}\{{g}^\Delta(\cdot)\}(\omega)=\frac{\Phi_{\Delta}(e^{i\omega\Delta})}{\Theta_{\Delta}(e^{i\omega\Delta})}\mathcal{F}\{g(\cdot)\}(\omega)=\sigma\frac{\Phi_{\Delta}(e^{i\omega\Delta})}{\Theta_{\Delta}(e^{i\omega\Delta})}\frac{b(-i\omega)}{a(-i \omega)},\quad \omega\in\bbr,\quad \Delta\g 0.$$

Since we can write $L_t-L_{t-\Delta}=\int^{t}_{-\infty}\bone_{(0,\Delta)}(t-s)dL_s$, the sum of the differences between the rescaled sampled noise terms and the increments of the L\'evy process is given by
\begin{align}
\sum_{j=1}^{n} \bar{L}^\Delta_j-L_{n\Delta} &
=\int^{n\Delta }_{-\infty}\sum_{j=1}^{n}\left[\frac{\sqrt{\Delta}}{\sigma_\Delta}g^\Delta(j\Delta-s)-\bone_{(0,\Delta)}(j\Delta-s)\right]dL_s 
 = \int_{-\infty}^{n\Delta} h^\Delta_n(n\Delta-s) dL_s,\label{diff:CMA}
\end{align}
where, for every $n\in\bbn$, $$h^\Delta_n(s):=\sum_{j=1}^{n}\left[\frac{\sqrt{\Delta}}{\sigma_\Delta}g^\Delta(s+(j-n)\Delta)-\bone_{(0,\Delta)}(s+(j-n)\Delta)\right],\quad s\in\bbr.$$ 
Note that the stochastic integral in Eq.~\eqref{diff:CMA} w.r.t. $L$ is still in the $L^2$-sense.
It is a standard result, cf. \cite[Ch. IV, \textsection 4]{Gikhmanetal2004}, that the variance of the moving average process in Eq.~\eqref{diff:CMA} is given by 
\begin{equation*} 
 \bbe\left[\sum_{j=1}^{n} \bar{L}^\Delta_j - L_{n\Delta}\right]^2 = \int_{-\infty}^{n\Delta}\left(h^\Delta_n(n\Delta-s)\right)^2 ds = \| h^\Delta_n(\cdot)\|_{L^2}^2,
\end{equation*}
where the latter equality is true since $h^\Delta_n(s) = 0$ for any $s\leq 0$. 

Furthermore, the Fourier transform of $h^\Delta_n(\cdot)$ can be readily calculated, invoking the linearity and the shift property of the Fourier transform. We thus obtain 
\begin{align}
 \mathcal{F}\{h^\Delta_n(\cdot)\}(\omega) &=\left[\frac{\sqrt{\Delta}}{\sigma_\Delta}\mathcal{F}\{g^\Delta(\cdot)\}(\omega)-\mathcal{F}\{\bone_{(0,\Delta)}(\cdot)\}(\omega)\right]\sum_{j=1}^{n}e^{i\omega(n-j)\Delta} \nonumber \\
  & = \left[\sigma\frac{\sqrt{\Delta}}{\sigma_\Delta}\dfrac{\prod_{j=1}^p(1-e^{\Delta(\lambda_j + i\omega)})}{\Theta_{\Delta}(e^{i\omega\Delta})}\frac{b(-i\omega)}{a(-i\omega)}-\frac{e^{i \omega \Delta}-1}{i\omega }\right]\frac{1-e^{i   \omega \Delta n}}{1-e^{i \omega\Delta }} \nonumber \\
  & =: \left[h^{\Delta, 1}(\omega) - h^{\Delta, 2}(\omega)\right]\cdot h_n^{\Delta, 3}(\omega),\quad \omega\in\bbr. \nonumber 
\end{align}
Due to Plancherel's Theorem, we deduce
\begin{align}
 \var&\left[\sum_{i=1}^{n} \bar{L}^\Delta_j - L_{n\Delta}\right] = \| h^\Delta_n(\cdot)\|_{L^2}^2 = \frac{1}{2\pi}\int_{\bbr}|\mathcal{F}\{h^\Delta_n(\cdot)\}|^2(\omega)d\omega, \nonumber \\
  & = \frac{1}{2\pi}\int_\bbr\left[\left|h^{\Delta, 1}\cdot h_n^{\Delta, 3}(\omega)\right|^2 + \left|h^{\Delta, 2}\cdot h_n^{\Delta, 3}(\omega)\right|^2 - 2\Re\left(h^{\Delta, 1}\cdot\overline{h^{\Delta, 2}}(\omega)\right)\left|h_n^{\Delta, 3}(\omega)\right|^2\right] d\omega. \label{var:spectral} 
\end{align}
It is easy to see that the first two integrals in Eq.~\eqref{var:spectral} are, respectively, the {variances} of $\sum_{i=1}^{n} \bar{L}^\Delta_j$ and $L_{n\Delta}$, both equal to $n\Delta$. 
Setting $n := \lfloor t/\Delta \rfloor$ yields for fixed positive $t$, as $\Delta\downarrow 0$,  
\begin{align} 
 \var\left[\sum_{i=1}^{\lfloor t/\Delta \rfloor} \bar{L}^\Delta_j - L_{\lfloor t/\Delta \rfloor\Delta}\right] &= 2\lfloor t/\Delta \rfloor\Delta - \frac{1}{\pi}\int_\bbr\Re\left(h^{\Delta, 1}\cdot\overline{h^{\Delta, 2}}(\omega)\right)\left|h_{\lfloor t/\Delta \rfloor}^{\Delta, 3}(\omega)\right|^2 d\omega \nonumber \\
  &= 2t(1+o(1)) - \frac{1}{\pi}\int_\bbr\Re\left(h^{\Delta, 1}\cdot\overline{h^{\Delta, 2}}(\omega)\right)\left|h_{\lfloor t/\Delta \rfloor}^{\Delta, 3}(\omega)\right|^2 d\omega. \nonumber
\end{align}
Hence, to show Eq.~\eqref{claim2}, it remains to prove that
\begin{equation*}
 \frac{1}{\pi}\int_\bbr\Re\left(h^{\Delta, 1}\cdot\overline{h^{\Delta, 2}}(\omega)\right)\left|h_{\lfloor t/\Delta \rfloor}^{\Delta, 3}(\omega)\right|^2 d\omega = 2t(1+o(1))\quad\text{ as }\Delta\downarrow 0, 
\end{equation*}
which in turn is equivalent to 
\begin{align} 
 \frac{1}{2\pi t}\int_\bbr\sigma&\frac{\sqrt{\Delta}}{\sigma_\Delta}\frac{1 - \cos(\omega\lfloor t/\Delta\rfloor \Delta)}{1 - \cos(\omega\Delta)}
  \Bigg[\frac{\sin(\omega\Delta)}{\omega}\Re\left(\dfrac{\prod_{j=1}^p(1-e^{\Delta(\lambda_j + i\omega)})}{\Theta_{\Delta}(e^{i\omega\Delta})}\frac{b(-i\omega)}{a(-i\omega)}\right) \nonumber \\
  & + \frac{1 - \cos(\omega\Delta)}{\omega}\Im\left(\dfrac{\prod_{j=1}^p(1-e^{\Delta(\lambda_j + i\omega)})}{\Theta_{\Delta}(e^{i\omega\Delta})}\frac{b(-i\omega)}{a(-i\omega)}\right)\Bigg] d\omega = 1+o(1)\quad\text{ as }\Delta\downarrow 0. \label{remaining statement}
\end{align}

Now, Lemma \ref{convergence:integrand} asserts that the integrand in Eq.~\eqref{remaining statement} converges pointwise, for every $\omega\neq 0$, to $2(1-\cos(\omega t))/\omega^2$ as $\Delta\downarrow 0$.
Since, for  sufficiently small $\Delta$, the integrand is dominated by an integrable function (see Lemma \ref{dominant:integrand}), we can apply Lebesgue's Dominated Convergence Theorem and deduce that the LHS of Eq.~\eqref{remaining statement}
converges, as $\Delta\downarrow 0$, to
$$ \frac{1}{\pi t}\int_\bbr \frac{1-\cos(\omega t)}{\omega^2} d\omega = \frac{2}{\pi}\int_0^\infty \frac{1 - \cos(\omega)}{\omega^2} d\omega = 1.$$
This proves \eqref{remaining statement} and concludes the proof of the {``if''-statement}.

{As to the ``only if''-part, let $J:=\{j=1,\ldots,q:\ \Re(\mu_j)<0\}$ and suppose that $|J|\geq 1$. 
Due to Eq.~\eqref{MA:factors} we have for $\Delta\downarrow0$
\begin{align}\frac{b(-i\omega)}{\Theta_\Delta(e^{i\omega\Delta})}&=\prod_{j=1}^{p-q-1}\left(1+\eta(\xi_j)\right)^{-1}\prod_{j=1}^q\frac{\mu_j-i\omega}{1-\zeta_j\,e^{i\omega\Delta}}\notag\\
&=\prod_{j=1}^{p-q-1}\left(1+\eta(\xi_j)\right)^{-1}\Delta^{-q}\prod_{j\in J}\frac{\mu_j-i\omega}{-\mu_j-i\omega}(1+o(1))\notag\\
&=\prod_{j=1}^{p-q-1}\left(1+\eta(\xi_j)\right)^{-1}\Delta^{-q}(1+D(\omega))(1+o(1)),\qquad\omega\in\bbr,\label{broke}\end{align}
where $D(\omega):=-1 + \prod_{j\in J}(\mu_j-i\omega)/(-\mu_j-i\omega)$. 
{ By virtue of Lemmata \ref{convergence:integrand} and \ref{dominant:integrand}}, we then obtain that the LHS of Eq.~\eqref{remaining statement} converges, as $\Delta\downarrow 0$, to 
$$ \frac{1}{\pi t}\int_\bbr \frac{1-\cos(\omega t)}{\omega^2}\big(1+\Re(D(\omega))\big)\, d\omega = 1 + \frac{1}{\pi}\int_\bbr\frac{1-\cos(\omega)}{\omega^2}\Re(D(\omega/t))\, d\omega.$$
Since $|\prod_{j\in J}(\mu_j-i\omega)/(-\mu_j-i\omega)|=1$, we further deduce that $\Re(D(\omega))\leq 0$ for any $\omega\in\bbr$. Obviously, $\Re(D(\omega))\not\equiv 0$ and hence, 
$$ \frac{1}{\pi t}\int_\bbr \frac{1-\cos(\omega t)}{\omega^2}\big(1+\Re(D(\omega))\big)\, d\omega < 1.$$
This shows that the convergence result \eqref{claim2} cannot hold. 

}
\end{proof}
In the following, we state three auxiliary results. For the proof of the first one, we need a concrete representation of the function $\alpha_n(x)$, which is defined in Eq.~\eqref{expansion:geometric:roots}. It can be shown 
that 
\begin{equation*}
 \alpha_{n}(x)=\frac{P_n(x)}{(2n+1)!\,x^{n+1}}, \quad x\neq0,\ n\in\bbn,
\end{equation*}
where $P_{n}(x)$ is a polynomial of order $n$ in $x$, namely
\begin{align}
P_n(x)=&\sum_{j=0}^{n} x^{n-j}\sum_{k=j+1}^{n} {(2k)!\stirlingtwo{2n+1}{2k}}\sum_{i=j}^{k}\left[ {i+1\choose j+1}\binom{2k}{2 i+1}-{i\choose j+1}\binom{2k}{2 i}\right](-2)^{j+1-2k}\nonumber\\
&+\sum_{j=0}^{n} x^{n-j}\sum_{k=j}^{n}{(2k+1)!\stirlingtwo{2n+1}{2k+1}}\sum_{i=j}^{k} \left[ {i+1\choose j+1}\binom{2k+1}{2 i+1}-{i\choose j+1}\binom{2k+1}{2 i}\right](-2)^{j-2k}\label{Poly},
\end{align}
with $\stirlingtwo{\cdot}{\cdot}$ being the Stirling number of the second kind. 
\begin{prop}\label{prop:real:roots}
All the zeroes of $\alpha_n(x)$ 
are real, distinct and greater than 2.
\end{prop}
\begin{proof} 
Using Eq.~\eqref{Poly}, we easily see that, for $P_n(x)=p_0+p_1x+\ldots+p_n x^n$,
\begin{equation}p_0=(-2)^{-n}(2n+1)!,\quad p_n=1\label{coeff:poly},\end{equation}
i.e. $P_n(x)$ will have $n$, potentially complex, roots, and they cannot be zero. 
Moreover, it is easy to verify that 
$$f(z,x) := {\frac{\sinh(z)}{\cosh(z)-1+x}} = \frac{e^{2z}-1}{e^{2z}+1+2(x-1)e^{z}},\quad z\in{\bbc},\ x\neq 0,$$
solves the mixed partial differential equation
\begin{equation}
\frac{\partial^2}{\partial z^2}f(z,x)=\left[(x-1) \frac{\partial}{\partial x}+x(x-2)\frac{\partial^2}{\partial x^2}\right]f(z,x). 
\label{PDE}
\end{equation}
We take $2n-1$ 
derivatives in $z$ on both sides of Eq.~\eqref{PDE}. Invoking the Schwarz Theorem, the product rule for derivatives and evaluating the resulting expression for $z=0$, we obtain that the function $\alpha_{n}(x)$ is given by recursion, for $x\not\in(0,2)$, as
\begin{align}
{(2n+3)\,(2n+1)}\,\alpha_{n+1}(x)&=\sqrt{x(x-2)}\frac{\partial}{\partial x}\left[\sqrt{x(x-2)}\frac{\partial}{\partial x}\alpha_n(x)\right],\label{recursion}\\
\alpha_{0}(x)&=1/x.\notag
\end{align}
We prove by induction 
that the roots 
are real, distinct and greater than 2. The functions ${\alpha_0(x)}=1/x$ and ${6\alpha_1(x)=(x-3)/x^2}$ have, respectively, no and one zero, so the claim can be partially verified. We start with $\alpha_2(x)=(30-15x+x^2)/{(120x^3)}$, whose zeroes are $\xi_{2,1}=1/2\left(15-\sqrt{105}\right)\approx 2.37652$ and $\xi_{2,2}=1/2\left(15+\sqrt{105}\right)\approx12.6235$, and note that they satisfy the claim. Assume that the statement is valid for $\alpha_n(x)$, $n\geq 2$, and its zeroes are $2<\xi_{n,1}<\xi_{n,2}<\ldots< \xi_{n,n}$. 

The derivative of {$\alpha_n(x)$} is of the form $Q_n(x)/x^{n+2}$, where ${(2n+1)!\,Q_n(x)}=x\frac{\partial}{\partial x}P_n(x)-(1+n)P_n(x)$. By virtue of {Rolle's Theorem}, $Q_n(x)$ has $n-1$ real roots $\chi_{n,i}$, $i=1,\ldots, n-1$, such that $2<\xi_{n,1}<\chi_{n,1}<\xi_{n,2}<\chi_{n,2}<\ldots< \chi_{n,n-1}< \xi_{n,n}$. Using the product rule and the value of the coefficients in Eq.~\eqref{coeff:poly}, we get
\begin{equation}
\frac{\partial}{\partial x} \alpha_n(x)\sim
-x^{-2}/(2n+1)!\rightarrow 0,\quad x\rightarrow \infty.\end{equation}
Again due to {Rolle's Theorem}, and since $\frac{\partial}{\partial x} \alpha_n(x)\rightarrow 0$ {and $\alpha_n(x)\to 0$} as $x\rightarrow \infty$, the function $Q_n(x)$ has a zero at some point $\xi_{n,n}<\chi_{n,n}<\infty$. For $x\geq 2$, the function {$\sqrt{x(x-2)}\frac{\partial}{\partial x}\alpha_n(x)$} is well defined {and it is zero for $x=2$ and $x=\chi_{n,i}, i=1,\ldots,n$}. 
{With the same arguments as before, we then obtain that $\frac{\partial}{\partial x}[\sqrt{x(x-2)}\frac{\partial}{\partial x}\alpha_n(x)]$ } is zero for $x=\xi_{n+1,i}$, $i=1,\ldots, n+1$, {where} $2<\xi_{n+1,1}<\chi_{n,1}<\xi_{n+1,2}<\chi_{n,2}<\ldots< \chi_{n,n}< \xi_{n+1,n+1}<\infty$. Due to Eq.~\eqref{recursion}, those zeroes are {also roots} of, respectively, $\alpha_{n+1}(x)$ and $P_{n+1}(x)$. Since $P_{n+1}(x)$ is a polynomial of order $n+1$, it  can have only $n+1$ roots, which were found already. Moreover, they are all real, distinct and strictly greater than 2, and the claim is proven.
\end{proof}

\begin{lem} \label{convergence:integrand}
 Suppose that {$\Re(\mu_j)\neq 0$} for all $j=1,\ldots,q$. We have, for any $t\in(0,\infty)$ and $\omega\neq 0$, 
 \begin{align*}
  \lim\limits_{\Delta\downarrow 0} \sigma\frac{\sqrt{\Delta}}{\sigma_\Delta}&\frac{1 - \cos(\omega\lfloor t/\Delta\rfloor \Delta)}{\omega}
  \frac{\sin(\omega\Delta)}{1 - \cos(\omega\Delta)}\Re\left(\dfrac{\prod_{j=1}^p(1-e^{\Delta(\lambda_j + i\omega)})}{\Theta_{\Delta}(e^{i\omega\Delta})}\frac{b(-i\omega)}{a(-i\omega)}\right) \\
   &= \frac{2 - 2\cos(\omega t)}{\omega^2}{  \big(1+\Re(D(\omega))\big)} 
 \end{align*}
 and
 $$\lim\limits_{\Delta\downarrow 0} \sigma\frac{\sqrt{\Delta}}{\sigma_\Delta}\frac{1 - \cos(\omega\lfloor t/\Delta\rfloor \Delta)}{\omega}\Im\left(\dfrac{\prod_{j=1}^p(1-e^{\Delta(\lambda_j + i\omega)})}{\Theta_{\Delta}(e^{i\omega\Delta})}\frac{b(-i\omega)}{a(-i\omega)}\right) = 0,$$
 where $D(\omega):=-1 + \prod_{j\in J}(\mu_j - i\omega)/(-\mu_j - i\omega)$ and $J:=\{j=1,\ldots,q:\ \Re(\mu_j)<0\}$. Obviously, if $\Re(\mu_j)>0$ for all $j=1,\ldots,q$, then 
 $D(\omega) = 0$ for all $\omega\in\bbr$.
\end{lem}

\begin{proof}
{Due to Proposition \ref{rootsgreater1}, we have that $\eta(\xi_j)\in(0,1)$ 
for sufficiently small $\Delta$.} 
Hence, for any $\omega\in\bbr$,
 \begin{align}
  \dfrac{\prod_{j=1}^p(1-e^{\Delta(\lambda_j + i\omega)})}{\Theta_{\Delta}(e^{i\omega\Delta})}\frac{b(-i\omega)}{a(-i\omega)} 
   &= \dfrac{1}{\prod_{j=1}^{p-q-1} (1 + \eta(\xi_j)e^{i\omega\Delta})}\prod\limits_{j=1}^p\frac{e^{\Delta(\lambda_j + i\omega)} - 1}{i\omega + \lambda_j} \prod\limits_{j=1}^q\frac{\mu_j-i\omega}{1-\zeta_j e^{i\omega\Delta}} \nonumber \\
   &= \Delta^{p-q}{  (1+D(\omega))}\prod\limits_{j=1}^{p-q-1} (1 + \eta(\xi_j))^{-1}\cdot(1+o(1))\quad\text{ as }\Delta\downarrow 0. \nonumber 
 \end{align}
 Moreover, using Eq.~\eqref{prediction:error}, we obtain
 $$ \sigma\frac{\sqrt{\Delta}}{\sigma_\Delta} = \frac{\sqrt{[2(p-q)-1]!\cdot\prod_{j=1}^{p-q-1}\eta(\xi_j)}}{\Delta^{p-q-1}}(1+o(1))\quad\text{ as }\Delta\downarrow 0. $$
 Since $\cos(\omega\lfloor t/\Delta\rfloor\Delta)\to \cos(\omega t)$ and $\Delta\sin(\omega\Delta)/(1 - \cos(\omega\Delta))\to 2/\omega$ as $\Delta\downarrow 0$ for any $\omega\neq 0$,
 we can use the equality (cf. \cite[proof of Theorem 3.2]{bfk:2011:2}) 
 $$ \frac{\sqrt{[2(p-q)-1]!\cdot\prod_{j=1}^{p-q-1}\eta(\xi_j)}}{\prod_{j=1}^{p-q-1} (1 + \eta(\xi_j))} = \frac{\prod_{j=1}^{p-q-1} |1 + \eta(\xi_j)|}{\prod_{j=1}^{p-q-1} (1 + \eta(\xi_j))}\cdot(1+o(1)) = 1+o(1)\quad\text{ as }\Delta\downarrow 0$$
 to conclude the proof. 
\end{proof}

\begin{lem} \label{dominant:integrand}
 Suppose that {$t\in(0,\infty)$} and $\Re(\mu_j)\neq0$ for all $j=1,\ldots,q$, and let the functions $h^{\Delta, 1}(\cdot), h^{\Delta, 2}(\cdot)$ and $h_{\lfloor t/\Delta\rfloor}^{\Delta, 3}(\cdot)$ be defined as in the proof of Theorem \ref{noise:extraction}. 
 There is a constant $C>0$ such that, for any $\omega\in\bbr$ and any sufficiently small $\Delta$,
 $$ \left|2\Re\left(h^{\Delta, 1}\cdot h_{\lfloor t/\Delta\rfloor}^{\Delta, 3}(\omega)\cdot\overline{h^{\Delta, 2}\cdot h_{\lfloor t/\Delta\rfloor}^{\Delta, 3}(\omega)}\right)\right|\leq 
 h(\omega),$$
 where $h(\omega) := \big(7^{2p}/2^{2p+q} + 1\big) t^2\bone_{(-1,1)}(\omega) + {C}{\omega^{-2}}\bone_{\bbr\backslash(-1,1)}(\omega)$. Moreover, $h$ is integrable over the real line.
\end{lem}

\begin{proof}
 We obviously have 
 \begin{equation} \label{obvious bound}
  \left|2\Re\left(h^{\Delta, 1}\cdot h_{\lfloor t/\Delta\rfloor}^{\Delta, 3}(\omega)\cdot\overline{h^{\Delta, 2}\cdot h_{\lfloor t/\Delta\rfloor}^{\Delta, 3}(\omega)}\right)\right|
  \leq\left|h^{\Delta, 1}\cdot h_{\lfloor t/\Delta\rfloor}^{\Delta, 3}(\omega)\right|^2 +  \left|h^{\Delta, 2}\cdot h_{\lfloor t/\Delta\rfloor}^{\Delta, 3}(\omega)\right|^2
 \end{equation}
 for any $\omega\in\bbr$ and any $\Delta$. Let us first consider the second addend on the RHS of Eq.~\eqref{obvious bound}.

 We obtain $|h^{\Delta, 2}\cdot h_{\lfloor t/\Delta\rfloor}^{\Delta, 3}(\omega)|^2 = 2(1 - \cos(\omega \lfloor t/\Delta\rfloor \Delta))/\omega^2$ and since $\lfloor t/\Delta\rfloor \Delta\leq t$ holds, 
 we can bound, for any $\Delta$, the latter function by $t^2$ on the interval $(-1,1)$ and by $4/\omega^2$ on $\bbr\backslash(-1,1)$. 

 As to the first addend on the RHS of Eq.~\eqref{obvious bound}, we calculate
 \begin{equation} \label{first addend}
  \left|h^{\Delta, 1}\cdot h_{\lfloor t/\Delta\rfloor}^{\Delta, 3}(\omega)\right|^2 
  = \sigma^2\frac{\Delta}{\sigma_\Delta^2}\frac{\prod_{j=1}^p \left|1 - e^{\Delta(\la_j + i\omega)}\right|^2}{\left|\Theta_\Delta(e^{i\omega\Delta})\right|^2}\frac{\left|b(-i\omega)\right|^2}{\left|a(-i\omega)\right|^2} \cdot \frac{1 - \cos(\omega\lfloor t/\Delta\rfloor \Delta)}{1 - \cos(\omega\Delta)}.
 \end{equation}
 
Let now $|\omega|<1$ and suppose that $\Delta$ is sufficiently small, i.e. the following inequalities are true for any $|\omega|<1$ whenever $\Delta$ is sufficiently small. 
 Using $|1 - e^z|\leq7/4|z|$ for $|z|<1$ (see, e.g., \cite[4.2.38]{abramowitz}) yields
 $$ \frac{\prod_{j=1}^p\left|1 - e^{\Delta(\la_j + i\omega)}\right|^2}{\left|a(-i\omega)\right|^2} \leq \left(\frac{7}{4}\Delta\right)^{2p}.$$
 The inequalities $(1-\cos(\omega\Delta))/(\omega\Delta)^2 \geq 1/4$ and $4(1-\cos(\omega\lfloor t/\Delta\rfloor\Delta))/\omega^2 \leq 2t^2$ (see above) imply
 $$ \frac{1 - \cos(\omega\lfloor t/\Delta\rfloor \Delta)}{1 - \cos(\omega\Delta)} \leq 2\left(\frac{t}{\Delta}\right)^2. $$
 As in the proof of Lemma \ref{convergence:integrand} we write $\Theta_{\Delta}(z) = \prod_{j=1}^{p-q-1} (1 + \eta(\xi_j)z)\cdot\prod_{j=1}^{q}(1-\zeta_j z),$
 where \linebreak $\zeta_j = 1 - {{\rm sgn}(\Re(\mu_j))}\,\mu_j\,\Delta + o(\Delta)$ (see \cite{bfk:2011:2}, Theorem 2.1). Since \linebreak $\prod_{j=1}^q\left(\left|1 - \zeta_j e^{i\omega\Delta}\right|/\Delta\right)^2\geq\prod_{j=1}^q {1}/{2}\left|{{\rm sgn}(\Re(\mu_j))}\,\mu_j - i\omega\right|^2$,
 we further deduce 
 $$ \frac{\left|b(-i\omega)\right|^2}{\prod_{j=1}^q\left|1 - \zeta_j e^{i\omega\Delta}\right|^2} \leq \frac{2^q}{\Delta^{2q}} .$$ 
 Again due to Eq.~\eqref{prediction:error}, we obtain
 \begin{align*}
  \sigma^2\frac{\Delta}{\sigma_\Delta^2}\prod\limits_{j=1}^{p-q-1}\left|1 + \eta(\xi_j)e^{i\omega\Delta}\right|^{-2} 
   &\leq \frac{2\cdot[2(p-q)-1]!}{\Delta^{2(p-q-1)}}\prod\limits_{j=1}^{p-q-1}\frac{\left|\eta(\xi_j)\right|}{\left|1 + \eta(\xi_j)e^{i\omega\Delta}\right|^2}
 \end{align*}
 and since $|\eta(\xi_j)|<1$ for all $j$ (see Proposition \ref{rootsgreater1}) we also have that $|1 + \eta(\xi_j)e^{i\omega\Delta}|\geq\frac{1}{2}|1 + \eta(\xi_j)|$ for all $j$, resulting in 
 \begin{align*}
  \sigma^2\frac{\Delta}{\sigma_\Delta^2}\prod\limits_{j=1}^{p-q-1}\left|1 + \eta(\xi_j)e^{i\omega\Delta}\right|^{-2} 
   &\leq\frac{2^{2(p-q)-1}}{\Delta^{2(p-q-1)}}\cdot[2(p-q)-1]!\prod\limits_{j=1}^{p-q-1}\frac{\left|\eta(\xi_j)\right|}{\left|1 + \eta(\xi_j)\right|^2} = \frac{2^{2(p-q)-1}}{\Delta^{2(p-q-1)}}.
 \end{align*}
 The latter equality follows from \cite[proof of Theorem 3.2]{bfk:2011:2}. All together the RHS of Eq.~\eqref{first addend} can be bounded for any $|\omega|<1$ and any sufficiently small $\Delta$ by
 $\left(7/2\right)^{2p} 2^{-q} t^2.$

 It remains to bound the RHS of Eq.~\eqref{first addend} also for $|\omega|\geq 1$. Hence, for the {rest of the} proof let us suppose $|\omega|\geq 1$ and in addition we assume again that $\Delta$ is sufficiently small. 
 We show that 
 $$ \sigma^2\frac{\Delta}{\sigma_\Delta^2}\frac{\prod_{j=1}^p \left|1 - e^{\Delta(\la_j + i\omega)}\right|^2}{\left|\Theta_\Delta(e^{i\omega\Delta})\right|^2}\frac{\left|b(-i\omega)\right|^2}{\left|a(-i\omega)\right|^2} \frac{1 - \cos(\omega\lfloor t/\Delta\rfloor \Delta)}{1 - \cos(\omega\Delta)}
  \leq\frac{C}{\omega^2} $$
 for some $C>0$. Since $\big|{\sigma^2\Delta}/\sigma_\Delta^2\big|\leq{\rm const.}\cdot\big|{\Delta^2}/{\Delta^{2(p-q)}}\big|$ (see \eqref{prediction:error}) and since
 $\prod_{j=1}^{p-q-1}|1 + \eta(\xi_j)e^{i\omega\Delta}|^{-2}\leq \prod_{j=1}^{p-q-1}(1 - |\eta(\xi_j)|)^{-2} \leq {\rm const.}$ {(cf. Proposition \ref{rootsgreater1})}, {it is sufficient} to prove 
 \begin{equation} \label{constant bound}
  \frac{(\omega\Delta)^2}{\Delta^{2(p-q)}}\frac{\prod_{j=1}^p \left|1 - e^{\Delta(\la_j + i\omega)}\right|^2}{\prod_{j=1}^q\left|1 - \zeta_je^{i\omega\Delta}\right|^2}\frac{\left|b(-i\omega)\right|^2}{\left|a(-i\omega)\right|^2} \frac{1 - \cos(\omega\lfloor t/\Delta\rfloor \Delta)}{1 - \cos(\omega\Delta)}
  \leq C
 \end{equation}
 for some $C>0$. {The power transfer function satisfies} $|b(-i\omega)|^2/|a(-i\omega)|^2\leq{{\rm const.}}/(\omega^{2(p-q)} + 1)$ for any $\omega\in\bbr$. Thus, Eq.~\eqref{constant bound} will follow from 
 \begin{equation} \label{constant bound:2}
  \frac{(\omega\Delta)^2}{(\omega\Delta)^{2(p-q)} + \Delta^{2(p-q)}}\frac{\prod_{j=1}^p \left|1 - e^{\Delta(\la_j + i\omega)}\right|^2}{\prod_{j=1}^q\left|1 - \zeta_je^{i\omega\Delta}\right|^2} \frac{1 - \cos(\omega\lfloor t/\Delta\rfloor \Delta)}{1 - \cos(\omega\Delta)}
  \leq C.
 \end{equation}
 We even show that Eq.~\eqref{constant bound:2} is true for any $\omega\in\bbr$. However, using symmetry and periodicity arguments it is sufficient to prove Eq.~\eqref{constant bound:2} on the interval $[0, \frac{2\pi}{\Delta}]$. 
 We split that interval into the following six subintervals 
 $$I_1 := \left[0, \min_{j=1,\ldots,q}\frac{|\mu_j|}{2}\right],\ I_2 := \left[\min_{j=1,\ldots,q}\frac{|\mu_j|}{2}, \max_{j=1,\ldots,q}2|\mu_j|\right],\
 I_3 := \left[\max_{j=1,\ldots,q}2|\mu_j|, \frac{\pi}{\Delta}\right],$$
 $$ I_4 := \left[\frac{\pi}{\Delta}, \frac{2\pi}{\Delta} - \max_{j=1,\ldots,q}2|\mu_j|\right],\
 I_5 := \left[\frac{2\pi}{\Delta} - \max_{j=1,\ldots,q}2|\mu_j|, \frac{2\pi}{\Delta} -\min_{j=1,\ldots,q}\frac{|\mu_j|}{2}\right]\text{ and }$$
 $$I_6 := \left[\frac{2\pi}{\Delta} -\min_{j=1,\ldots,q}\frac{|\mu_j|}{2}, \frac{2\pi}{\Delta}\right].$$ 

 For any $\omega\in I_1 \cup I_6$, the fraction $\frac{1 - \cos(\omega\lfloor t/\Delta\rfloor \Delta)}{1 - \cos(\omega\Delta)}$ can be bounded by $\lfloor t/\Delta\rfloor^2$.  
 In the other intervals we have the obvious bound $\frac{2}{1-\cos(\omega\Delta)}$ for that term. 

 Now, for any $j=1,\ldots,p$, we have, as $\Delta\downarrow 0$,  
 \begin{align*}
  \left|1-e^{\Delta\la_j}\cdot e^{i\omega\Delta}\right|^2 &\leq {2\left|1-e^{i\omega\Delta}\right|^2 + 4\Delta^2\left|\la_j\right|^2} = {8}\sin^2\left(\frac{\omega\Delta}{2}\right) + 4\Delta^2\left|\la_j\right|^2 
    \leq 4\Delta^2\left(\omega^2 + \left|\la_j\right|^2\right) 
 \end{align*}
 if $\omega\in I_1\cup I_2\cup I_3$, and $\left|1-e^{\Delta\la_j}\cdot e^{i\omega\Delta}\right|^2 \leq 4\Delta^2\left(\left({2\pi}/\Delta - \omega\right)^2 + \left|\la_j\right|^2\right)$ 
 if $\omega\in I_4\cup I_5\cup I_6$. 

 The first fraction on the LHS of Eq.~\eqref{constant bound:2} satisfies
 $$ \frac{(\omega\Delta)^2}{(\omega\Delta)^{2(p-q)} + \Delta^{2(p-q)}} \leq 
 \left\{ \begin{array}{l} \min\limits_{j=1,\ldots,q}\frac{|\mu_j|}{2}\cdot\frac{\Delta^2}{\Delta^{2(p-q)}},\ \hfill\text{if }\omega\in I_1, \\ \frac{(\omega\Delta)^2}{(\omega\Delta)^{2(p-q)}},\ \hfill\text{if }\omega\in I_2\cup I_3, \\ \frac{(2\pi)^2}{\pi^{2(p-q)}},\ \hfill\text{if }\omega\in I_4\cup I_5\cup I_6. \end{array} \right. $$
 
 Then, for any $j=1,\ldots,q$ and $\omega\in I_1\cup I_6$, we obtain
 \begin{align*}
  \left|1 - \zeta_je^{i\omega\Delta}\right|^2 &= \left|1 - (1-{{\rm sgn}(\Re(\mu_j))}\,\mu_j\Delta + o(\Delta))e^{i\omega\Delta}\right|^2 \geq \frac{1}{2}\Delta^2\left|{{\rm sgn}(\Re(\mu_j))}\,\mu_j - i\omega\right|^2 \\
   &\geq \frac{1}{8}\Delta^2\left|\mu_j\right|^2. 
 \end{align*}
 If $\omega\in I_3$, we have 
 \begin{align*}
  \left|1 - \zeta_je^{i\omega\Delta}\right|^2 &\geq \left(\left|1 - e^{i\omega\Delta}\right| - \left|\mu_j + o(1)\right|\Delta\right)^2 = \left(2\sin\left(\frac{\omega\Delta}{2}\right) - \left|\mu_j + o(1)\right|\Delta\right)^2 \\
   &\geq \Delta^2\left(\frac{3}{5}\omega - \left|\mu_j + o(1)\right|\right)^2
 \end{align*}
 and likewise, for $\omega\in I_4$, we deduce $\left|1 - \zeta_je^{i\omega\Delta}\right|^2 \geq \Delta^2\left(\frac{3}{5}(\frac{2\pi}{\Delta}-\omega) - \left|\mu_j + o(1)\right|\right)^2$. For $\omega\in I_2$ we get with arbitrary $\varepsilon>0$
 \begin{align*}
  \left|1 - \zeta_je^{i\omega\Delta}\right|^2 &= 2(1-\cos(\omega\Delta))\cdot(1-\Delta\,{{\rm sgn}(\Re(\mu_j))}\,\Re(\mu_j)+o(\Delta)) \\
  &\qquad\qquad\qquad + 2\sin(\omega\Delta)\cdot(-\Delta\,{{\rm sgn}(\Re(\mu_j))}\,\Im(\mu_j) + o(\Delta)) + \Delta^2|\mu_j|^2 + o(\Delta^2) \\
   & \geq (\omega\Delta)^2 \cdot (1-\varepsilon) - 2(\omega\Delta)\cdot\Delta\left|\Im(\mu_j)\right|\cdot(1+\varepsilon) + \Delta^2\left(|\mu_j|^2 + o(1)\right) =: f_\varepsilon^\Delta(\omega\Delta).
 \end{align*}
 Since $f_\varepsilon^\Delta(\omega)/\omega^2 \to 1-\varepsilon\ (\omega\to\infty)$ and $f_\varepsilon^\Delta(\omega)/\omega^2 \to\infty\ (\omega\to 0)$, a (global) minimum of $f_\varepsilon^\Delta(\omega)/\omega^2$ on $(0,\infty)$ could 
 be achieved in any $\omega^*$ with $\big(\frac{{\rm d}}{{\rm d}\omega}\frac{f_\varepsilon^\Delta(\omega)}{\omega^2}\big)(\omega^*) = 0$. The only such value is 
 $\omega^* = \frac{\Delta(|\mu_j|^2 + o(1))}{(1+\varepsilon)|\Im(\mu_j)|}$. Since 
 $$\frac{f_\varepsilon^\Delta(\omega^*)}{(\omega^*)^2} = 1-\varepsilon - (1+\varepsilon)^2\frac{|\Im(\mu_j)|^2}{|\mu_j|^2 + o(1)} \geq (1+\varepsilon)\frac{\Re(\mu_j)^2}{|\mu_j|^2} - 3\varepsilon -\varepsilon^2 \geq \frac{1}{2}\frac{\Re(\mu_j)^2}{|\mu_j|^2}$$ for, e.g., $\varepsilon = \frac{1}{6}\frac{\Re(\mu_j)^2}{|\mu_j|^2}$, we obtain
 $\frac{f_\varepsilon^\Delta(\omega)}{\omega^2} \geq \frac{1}{2}\frac{\Re(\mu_j)^2}{|\mu_j|^2}$ for any $\omega\in(0,\infty)$. 
 Hence, 
 $$ \left|1 - \zeta_je^{i\omega\Delta}\right|^2 \geq f_\varepsilon^\Delta(\omega\Delta) \geq \frac{1}{2}\frac{\Re(\mu_j)^2}{|\mu_j|^2}(\omega\Delta)^2\quad\text{ for all }\omega\in I_2.$$
 Using periodic properties of the sine and cosine terms, we likewise get 
 $$ \left|1 - \zeta_je^{i\omega\Delta}\right|^2 \geq \frac{1}{2}\frac{\Re(\mu_j)^2}{|\mu_j|^2}\Delta^2\left(\frac{2\pi}{\Delta} - \omega\right)^2\quad\text{ for any }\omega\in I_5.$$
 
 Putting all together, we can bound the LHS of Eq.~\eqref{constant bound:2} in $I_1$ by
 \begin{align*}
  \min\limits_{j=1,\ldots,q}\frac{|\mu_j|}{2}\cdot\frac{(\lfloor t/\Delta\rfloor\Delta)^2}{\Delta^{2(p-q)}}&\frac{4^p\Delta^{2p}\cdot\prod_{j=1}^p \big({\min_{k=1,\ldots,q}|\mu_k|^2/4} + |\la_j|^2\big)}{8^{-q}\Delta^{2q}\prod_{j=1}^q|\mu_j|^2} \\ 
   &\leq \min\limits_{j=1,\ldots,q}\frac{|\mu_j|}{2}\cdot t^2\cdot\frac{4^{p+q}\cdot\prod_{j=1}^p \big({\min_{k=1,\ldots,q}|\mu_k|^2/4} + |\la_j|^2\big)}{\prod_{j=1}^q\frac{1}{2}|\mu_j|^2} = C,
 \end{align*}
 in $I_2$ by
 \begin{align*}
  \frac{2(\omega\Delta)^2}{1-\cos(\omega\Delta)}&\frac{4^p\Delta^{2p}\cdot\prod_{j=1}^p \left(4\max_{k=1,\ldots,q}|\mu_k|^2 + |\la_j|^2\right)}{(\omega\Delta)^{2p}\cdot\prod_{j=1}^q\frac{1}{2}\frac{\Re(\mu_j)^2}{|\mu_j|^2}} \\
   &\leq \frac{5\cdot 4^{2p}\cdot\prod_{j=1}^p \left(4\max_{k=1,\ldots,q}|\mu_k|^2 + |\la_j|^2\right)}{\min_{j=1,\ldots,q}|\mu_j|^{2p}\cdot\prod_{j=1}^q\frac{1}{2}\frac{\Re(\mu_j)^2}{|\mu_j|^2}} = C,
 \end{align*}
 in $I_3$ by 
 \begin{align*}
  \frac{2(\omega\Delta)^2}{1-\cos(\omega\Delta)}&\frac{{4^p}(\omega\Delta)^{2p}\cdot\prod_{j=1}^p \left(1+\frac{|\la_j|^2}{4\max_{k=1,\ldots,q}|\mu_k|^2}\right)}{(\omega\Delta)^{2(p-q)}\cdot(\frac{1}{20}\omega\Delta)^{2q}} \\
   &\leq \pi^2\,{4^{p}}\,20^{2q}\prod_{j=1}^p \left(1+\frac{|\la_j|^2}{4\max_{k=1,\ldots,q}|\mu_k|^2}\right) = C,
 \end{align*}
 in $I_4$ by
 \begin{align*}
  \frac{(2\pi)^2}{\pi^{2(p-q)}}\frac{2}{1-\cos(\omega\Delta)}&\frac{{4^p}(2\pi - \omega\Delta)^{2p}\cdot\prod_{j=1}^p \left(1+\frac{|\la_j|^2}{4\max_{k=1,\ldots,q}|\mu_k|^2}\right)}{20^{-2q}\,(2\pi - \omega\Delta)^{2q}} \\
   & \leq {4^{p+1}}\,20^{2q}\prod_{j=1}^p \left(1+\frac{|\la_j|^2}{4\max_{k=1,\ldots,q}|\mu_k|^2}\right)\frac{{2}\cdot(2\pi - \omega\Delta)^{2}}{1-\cos(2\pi - \omega\Delta)} \\
   & \leq \pi^2\,{4^{p+1}}\,20^{2q}\prod_{j=1}^p \left(1+\frac{|\la_j|^2}{4\max_{k=1,\ldots,q}|\mu_k|^2}\right) = C,
 \end{align*}
 in $I_5$ by
 \begin{align*}
  \frac{(2\pi)^2}{\pi^{2(p-q)}}\frac{2}{1-\cos(\omega\Delta)}&\frac{4^p\Delta^{2p}\cdot\prod_{j=1}^p \left(4\max_{k=1,\ldots,q}|\mu_k|^2 + |\la_j|^2\right)}{\Delta^{2q}\prod_{j=1}^q\frac{1}{8}\min_{k=1,\ldots,q}|\mu_k|^2(\Re(\mu_j)/|\mu_j|)^2} \\
   &\leq \frac{(2\pi)^2}{\pi^{2(p-q)}}\frac{4^p\cdot\prod_{j=1}^p \left(4\max_{k=1,\ldots,q}|\mu_k|^2 + |\la_j|^2\right)}{\prod_{j=1}^q\frac{1}{8}\min_{k=1,\ldots,q}|\mu_k|^2(\Re(\mu_j)/|\mu_j|)^2}\frac{2\Delta^2}{1-\cos(2\pi - \omega\Delta)} \\
   &\leq \frac{(2\pi)^2}{\pi^{2(p-q)}}\frac{4^p\cdot\prod_{j=1}^p \left(4\max_{k=1,\ldots,q}|\mu_k|^2 + |\la_j|^2\right)}{\prod_{j=1}^q\frac{1}{8}\min_{k=1,\ldots,q}|\mu_k|^2(\Re(\mu_j)/|\mu_j|)^2}\frac{5\cdot 4}{\min_{j=1,\ldots,q}|\mu_j|^2} = C,
 \end{align*}
 and, finally, in $I_6$ by
 \begin{align*}
  \frac{(2\pi\lfloor t/\Delta\rfloor)^2}{\pi^{2(p-q)}}&\frac{4^p\Delta^{2p}\cdot\prod_{j=1}^p \big({\min_{k=1,\ldots,q}|\mu_k|^2/4} + |\la_j|^2\big)}{8^{-q}\Delta^{2q}\prod_{j=1}^q|\mu_j|^2} \\
  & \leq \frac{(2\pi t)^2}{\pi^{2(p-q)}}\frac{4^{p+q}\cdot\prod_{j=1}^p \left({\min_{k=1,\ldots,q}|\mu_k|^2/4} + |\la_j|^2\right)}{\prod_{j=1}^q\frac{1}{2}|\mu_j|^2} = C.
 \end{align*}
 This shows Eq.~\eqref{constant bound:2} and thus concludes the proof.

\end{proof}

\end{appendix}
\subsubsection*{Acknowledgments}
The two authors take pleasure in thanking their PhD advisors Vicky Fasen, Claudia Kl\"uppelberg and Robert Stelzer for helpful comments, discussions and careful proofreading. Moreover, the authors are grateful to Peter Brockwell for comments on previous drafts. The work of V.F. was supported by the International Graduate School of Science and Engineering (IGSEE) of the Technische Universit\"at M\"unchen. Financial support for F.F. by the Deutsche Forschungsgemeinschaft through the research grant STE 2005/1-1 is gratefully acknowledged. We are also indebted to an associate editor and a referee for valuable comments. 
\bibliographystyle{kluwer.bst}

\end{document}